\documentclass[12pt, a4paper, UKenglish]{article}

\usepackage[UKenglish]{babel}

\usepackage{amssymb,amsmath,amsfonts,amsthm,nomencl,mathrsfs } 
\usepackage{bbm}
\usepackage[arrow, matrix, curve]{xy}
\usepackage[latin1]{inputenc}
\usepackage{a4wide}
\usepackage{color}
\usepackage{hyperref}
\usepackage{upref}
\usepackage{dirtytalk}

\newcommand{\IC}{\mathbb{C}}
\newcommand{\IR}{\mathbb{R}}

\renewcommand{\H}{ H}

\newcommand{\question}[1]{\leavevmode{\marginpar{\tiny%
			$\hbox to 0mm{\hspace*{-0.5mm}$\leftarrow$\hss}%
			\vcenter{\vrule depth 0.1mm height 0.1mm width \the\marginparwidth}%
			\hbox to 0mm{\hss$\rightarrow$\hspace*{-0.5mm}}$\\\relax\raggedright #1}}}

\newcommand{\rg}{\mathrm{rg}}

\newcommand{\tr}{\mathrm{tr}}

\newcommand{\dom}{\mathrm{Dom}}

\newcommand{\IN}{\mathbb{N}}
\newcommand{\IZ}{\mathbb{Z}}

\newcommand{\Id}{{\rm d}}

\newcommand{\p}{\partial}

\newcommand{\one}{\mathbbm{1}}
\DeclareMathOperator*{\esssup}{ess\,sup}
\DeclareMathOperator*{\supp}{supp}

\theoremstyle{plain}            
\newtheorem{theorem}{theorem}[section]
\newtheorem{Lemma}[theorem]{Lemma}
\newtheorem{Corollary}[theorem]{Corollary}
\newtheorem{Theorem}[theorem]{Theorem}
\newtheorem{Proposition}[theorem]{Proposition}

\theoremstyle{definition}
\newtheorem{Definition}[theorem]{Definition}

\newtheorem{Hypothesis}[theorem]{Hypothesis}
\newtheorem{Remark}[theorem]{Remark}
\newtheorem{Example}[theorem]{Example}

\title{Trace Class Properties of Resolvents of Callias Operators}
\author{Oliver F\"urst}
\date{28th June 2022}

\begin{document}
	
	\maketitle
	
	\begin{abstract}
		We present conditions for a family $\left(A\left(x\right)\right)_{x\in\IR^{d}}$ of self-adjoint operators in $H^{r}=\IC^{r}\otimes H$ for a separable complex Hilbert space $H$, such that the Callias operator $D=ic\nabla+A\left(X\right)$ satisfies that $\left(D^{\ast}D+1\right)^{-N}-\left(DD^{\ast}+1\right)^{-N}$ is trace class in $L^2\left(\IR^{d},H^{r}\right)$. Here, $c\nabla$ is the Dirac operator associated to a Clifford multiplication $c$ of rank $r$ on $\IR^{d}$, and $A\left(X\right)$ is fibre-wise multiplication with $A\left(x\right)$ in $L^2\left(\IR^{d},H^{r}\right)$.
\end{abstract}
	
\section{Introduction}

	Let $d,r\in\IN$ and let $c$ be a Clifford multiplication over $\IR^{d}$ in $\IC^{r}$, and consider the associated Dirac operator $c\nabla$ on $\IR^{d}$,
	\begin{align}
		c\nabla f=\sum_{j=1}^{d}c\left(\Id x^{j}\right)\p_{x^{j}}f,\ f\in C^{1}\left(\IR^{d},\IC^{r}\right).
	\end{align}
	
	Let $H$ be a separable complex Hilbert space. Denote $H^{r}:=\IC^{r}\otimes H$. We consider $c\nabla$ as a self-adjoint linear operator in $L^2\left(\IR^{d},H^{r}\right)$ with domain $\dom\left(c\nabla\right):=W^{1,2}\left(\IR^{d},H^{r}\right)$, the $L^2$-Sobolev space of order $1$ with values in $H^{r}$.
	
	Now let $A=\left(A\left(x\right)\right)_{x\in\IR^{d}}$ be a family of self-adjoint operators in $H^{r}$. Then we associate to $A$ the fibre wise multiplication operator $A\left(X\right)$ in $L^2\left(\IR^{d},H^{r}\right)$, given by
	\begin{align}
		\left(A\left(X\right)f\right)\left(x\right):=&A\left(x\right)f\left(x\right),\ x\in\IR^{d}.
	\end{align}
	The Callias operator $D$ is then
	\begin{align}
		D:=ic\nabla+A\left(X\right).
	\end{align}
	The goal of this paper is to give conditions on the operator family $A$, such that the resolvent powers of $D^{\ast}D$ and $DD^{\ast}$ are trace comparable, i.e. such that for some $N\in\IN$,
	\begin{align}\label{caltracecomp}
		\left(D^{\ast}D+1\right)^{-N}-\left(DD^{\ast}+1\right)^{-N}\in S^1\left(L^2\left(\IR^{d},H^{r}\right)\right),
	\end{align}
	where $S^{p}\left(Y\right)$ denotes the $p$-th Schatten-von Neumann operators on a Hilbert space $Y$.
	
Before we proceed with a proper introduction of the involved operators, let us briefly contextualize why one might be interested in property (\ref{caltracecomp}).

The name "Callias operator" reflects the contribution of C. Callias in \cite{Cal} to the Fredholm index problem of the operator $D$ for a family of matrix potentials $A$, which arises from Yang-Mills theory. A similar index problem in the special case of dimension $d=1$ had been considered in the seminal series of articles by  M. Atiyah, V. Patodi, and I. Singer in \cite{APS1},\cite{APS2}, and especially \cite{APS3}, dealing with the index problem on manifolds with boundary. In their setup, the family $A$ is given by first order differential operators on a compact manifold. 

A. Pushnitski showed in \cite{Push} that the Callias operator $D$ in one dimension $d=1$ for a family $A\left(x\right)=A_{-}+B\left(x\right)$, $x\in\IR$, where
\begin{align}\label{pushcond}
	\int_{\IR}\left\|B'\left(x\right)\right\|_{S^1\left(H\right)}\Id x<\infty,
\end{align}
satisfies condition (\ref{caltracecomp}) for $N=1$, and one obtains the trace formula
\begin{align}\label{pushtraceformula}
	\tr_{L^2\left(\IR,H\right)}\left(\left(D^{\ast}D+z\right)^{-1}-\left(DD^{\ast}+z\right)^{-1}\right)&=\frac{1}{2z}\tr_{H}\left(g_{z}\left(A_{+}\right)-g_{z}\left(A_{-}\right)\right),\ z\geq 1,\nonumber\\
	g_{z}\left(\lambda\right)&:=\frac{\lambda}{\left(\lambda^{2}+z\right)^{\frac{1}{2}}},
\end{align}
where the limits $A_{\pm}=\lim_{x\to\pm\infty}A\left(x\right)$ exist in an appropriate sense. Furthermore it was shown that one can calculate the Fredholm index of $D$ in terms of the spectral shift function of the operators $A_{+}$ and $A_{-}$.

The results by Pushnitski were shown under less restrictive conditions on the family $\left(A\left(x\right)\right)_{x\in\IR}$ by F. Gesztesy et al. in \cite{GLMST}. The authors were especially able to weaken the required condition (\ref{pushcond}) to
\begin{align}\label{gescond}
	\int_{\IR}\left\|B'\left(x\right)\left(A_{-}^{2}+1\right)^{-\frac{1}{2}}\right\|_{S^1\left(H\right)}\Id x<\infty,
\end{align}
which is the first instance where the family of operators $A\left(x\right)$ are generated by possibly unbounded, non-commutative and non-discrete perturbations.

More recently, A. Carey et al. in \cite{Car} showed property (\ref{caltracecomp}) holds for an appropriate $N\in\IN$ in the one-dimensional case $d=1$, if the family $A$ is itself given by perturbations of a Dirac operator by a matrix potential, and if one imposes weaker trace-class requirements on $A$ than (\ref{gescond}), dependent on a parameter (on which the choices for $N$ also depend).

The goal of this paper is to present a collection of conditions on the family $A$ for general $d\in\IN$, such that $A$ can be obtained by unbounded, non-commutative and non-discrete perturbations of a general self-adjoint model operator $A_{0}$, such that the associated Callias operator $D$ satisfies property (\ref{caltracecomp}) for a suitable $N\in\IN$. The necessary requirements on the trace-class properties of $A$ will be weaker for larger $N$. To determine the trace and index formula for $D$ in this instance will however be a task for future work.

Let us present the main result of this paper and review the conditions on the operator family $A$.

We start with reasonably generic assumptions, which allow us to differentiate the operator family $A$ on $\IR^{d}$. These conditions are chosen such that one obtains favourable domain properties of $D$, and they do not pertain to the trace-class properties directly.

Let $A_{0}$ be a self-adjoint operator in $H^{r}$ and let $A\left(x\right)$, $x\in\IR^{d}$, be symmetric operators, with the dense set
\begin{align}
	\mathcal{D}=\bigcup_{n\in\IN}\rg\left(\one_{\left[-n,n\right]}\left(A_{0}\right)\right)\subseteq H^{r},
\end{align}
contained in all domains $\dom\left(A\left(x\right)\right)$, $x\in\IR^{d}$. For compatibility with Clifford multiplication we require that $c\left(\Id x^{j}\right)\left(\mathcal{D}\right)\subseteq\dom\left(A_{0}\right)$, $j\in\left\{1,\ldots,d\right\}$. We assume that for all $\phi\in\mathcal{D}$ and $\psi\in H^{r}$ the function
\begin{align}
	x\mapsto\langle A\left(x\right)\phi,\psi\rangle_{H^{r}},
\end{align}
is Lebesgue measurable on $\IR^{d}$ (see Hypothesis \ref{basichyp}). For a yet to choose $N\in\IN$, we then require that $x\mapsto A\left(x\right)\phi$ is $2N-1$-times weakly differentiable for $\phi\in\mathcal{D}$, with derivatives in $L^2_{loc}\left(\IR^{d},H^{r}\right)$, which we abbreviate with $A\in W^{2N-1,2,End}_{loc}\left(\IR^{d},\left(\mathcal{D},H^{r}\right)\right)$  (see Definition \ref{opdiffdef}). We impose Kato-Rellich type bounds on the derivatives of $A$ to ensure that $\left(D^{\ast}D\right)^{N}$ and $\left(DD^{\ast}\right)^{N}$ are self-adjoint operators on the domain $W^{2N,2}\left(\IR^{d},H^{r}\right)\cap L^2\left(\IR^{d},\dom\left(A_{0}^{2N}\right)\right)$ (see Proposition \ref{hdomainprop}).

Let us now discuss the trace-class conditions on $A$, which are the essential assumptions.

Denote $n:=\max\left(\lfloor\frac{d}{2}-2\rfloor+1,0\right)\in\IN$, and let $\alpha,\beta\in\IR^{\geq 0}$. If $A\left(x\right)$ commutes with $c\left(\Id x^{j}\right)$ for all $x\in\IR^{d}$ and $j\in\left\{1,\ldots,d\right\}$, then let $N\in\IN$ with $N>\frac{\alpha}{2}+\frac{n}{2}+\frac{d}{4}$, otherwise let $N>\max\left(\frac{\alpha}{2}+\frac{n}{2}+\frac{d}{4},\frac{\beta}{2}+\frac{n}{2}+\frac{d}{4}+\frac{1}{2}\right)$.

Denote
\begin{align}
 	c\nabla^{End}A&:=\sum_{j=1}^{d}c\left(\Id x^{j}\right)\p_{x^{j}}^{End}A,\nonumber\\
 	\nabla^{End}Ac&:=\sum_{j=1}^{d}\p_{x^{j}}^{End}Ac\left(\Id x^{j}\right),
 \end{align}
where $\left(\p_{x^{j}}^{End}A\right)f=\p_{x^{j}}\left(Af\right)-A\p_{x^{j}}f$ are appropriately defined partial derivatives of operators (see Definition \ref{opdiffdef}). Then we assume
\begin{align}\label{tracecondeq1}
	\int_{\IR^{d}}\left\|\p^{\gamma,End}\left(c\nabla^{End}A\langle A_{0}\rangle^{-\alpha}\right)\left(x\right)\right\|_{S^1\left(H^{r}\right)}\Id x&<\infty,\nonumber\\
	\int_{\IR^{d}}\left\|\p^{\gamma,End}\left(\nabla^{End}Ac\langle A_{0}\rangle^{-\alpha}\right)\left(x\right)\right\|_{S^1\left(H^{r}\right)}\Id x&<\infty,\nonumber\\
	\int_{\IR^{d}}\left\|\p^{\gamma,End}\left(\left[c\left(\Id x^{j}\right),A\right]\langle A_{0}\rangle^{-\beta}\right)\left(x\right)\right\|_{S^1\left(H^{r}\right)}\Id x&<\infty,\ j\in\left\{1,\ldots,d\right\},
\end{align}
for all $\gamma\in\IN^{d}$ with $\left|\gamma\right|\leq n$.

We should stress that $n=0$ for dimensions $d\leq 3$, which contains the situation discussed in \cite{Push},\cite{GLMST}, and \cite{Car}. We note especially that the first line of (\ref{tracecondeq1}) is analogous to the condition in the one-dimensional case for example like (\ref{gescond}) in \cite{GLMST}. In case $d=r=1$ the second and third line of (\ref{tracecondeq1}) is trivially satisfied since (scalar) $c$ commutes with $A$ automatically. In dimensions $d\geq 4$ we see that $N>1$ becomes necessary. In any case, we should note the basic feature of all the conditions in (\ref{tracecondeq1}): One can increase $\alpha$ and $\beta$ to decrease the restriction on the family $A$, in exchange for a possibly larger $N$, which weakens property (\ref{caltrace}) of the Callias operator $D$. This is for example useful if $A$ is a family of (pseudo-)differential operators (cf. Chapter 4 in \cite{Simon}). We will illustrate this last point in Example \ref{example} for the concrete case of a Dirac operator perturbed by a potential.

All presented conditions on $A$ together (see Hypothesis \ref{mainassumpt}) imply the desired property
\begin{align}\label{caltrace}
	\left(D^{\ast}D+1\right)^{-N}-\left(DD^{\ast}+1\right)^{-N}\in S^1\left(L^2\left(\IR^{d},H^{r}\right)\right),
\end{align}
for the Callias operator $D$, which is the principal result of this paper (see Theorem \ref{mainthm}).

\section{Notation and Basic Definitions}

	We fix $i$ as the imaginary unit in $\IC$. Let $H$ be a separable complex Hilbert space and denote with $H^{r}:=H\otimes\IC^{r}$ for $r\in\IN$. Let $L^{p}\left(\IR^{d}\right)$ be the space of $p$-Lebesgue-integrable elements on $\IR^{d}$, and $W^{k,2}\left(\IR^{d}\right)$ the $L^2$-Sobolev space of order $k$ on $\IR^{d}$. For $1\leq p\leq\infty$ denote with $L^{p}\left(\IR^{d},H^{r}\right)$ the $H^{r}$-valued (Bochner-Lebesgue-)$L^p$-elements over $\IR^{d}$, i.e. $f\in L^{p}\left(\IR^{d},H^{r}\right)$ if for all $\phi\in H^{r}$,
	\begin{align}
		\IR^{d}\ni x&\mapsto\langle f\left(x\right),\phi\rangle_{H^{r}}\ \text{is Lebesgue measurable},\nonumber\\
		\left\|f\right\|_{L^{p}\left(\IR^{d},H^{r}\right)}^{p}&:=\int_{\IR^{d}}\left\|f\left(x\right)\right\|_{H^{r}}^{p}\Id x<\infty,\ p<\infty,\nonumber\\
		\left\|f\right\|_{L^{\infty}\left(\IR^{d},H^{r}\right)}&:=\esssup_{x\in\IR^{d}}\left\|f\left(x\right)\right\|_{H^{r}}<\infty,\ p=\infty.
	\end{align}
	Let $\mathcal{F}$ be the (isometric) Fourier transform
	\begin{align}
		\mathcal{F}: L^2\left(\IR^{d},H^{r},\Id x\right)&\rightarrow L^2\left(\IR^{d},H^{r},\frac{1}{\left(2\pi\right)^{d}}\Id\xi\right),\nonumber\\
		\mathcal{F}\left(f\right)\left(\xi\right)&:=\int_{\IR^{d}}e^{-i\langle x,\xi\rangle_{\IR^{d}}}f\left(x\right)\Id x,
	\end{align}
	obtained via continuous extension and Hilbert space tensor product from the classical Fourier transform on Schwartz functions in $\IR^{d}$. Usually we write $\widehat{f}:=\mathcal{F}\left(f\right)$. Note that as in the scalar case the Fourier transform satisfies by complex interpolation for $f\in L^{p}\left(\IR^{d},H^{r}\right)$, $1\leq p\leq 2$, and $\frac{1}{p}+\frac{1}{q}=1$,
	\begin{align}
		\left\|\widehat{f}\right\|_{L^{q}\left(\IR^{d},H^{r}\right)}\leq\left\|f\right\|_{L^{p}\left(\IR^{d},H^{r}\right)}.
	\end{align}
	
	We also use the abbreviations $\langle z\rangle:=\left(1+\left|z\right|^{2}\right)^{\frac{1}{2}}$, $\one_{M}\left(z\right):=\begin{cases}
		1,&z\in M,\\
		0,&z\notin M,
	\end{cases}$
	for $M\subseteq\IC$, and $z\in\IC$, and $\mathcal{B}\left(V\right)$ for the Borel $\sigma$-algebra of a topological space $V$.
	
	Denote with $W^{k,2}\left(\IR^{d},H^{r}\right)$ the $H^{r}$-valued $L^2$-Sobolev space of order $k\geq 0$ over $\IR^{d}$, i.e. $f\in W^{k,2}\left(\IR^{d},H^{r}\right)$ if $f\in L^{2}\left(\IR^{d},H^{r}\right)$, and
	\begin{align}
		\IR^{d}\ni\xi\mapsto\langle\xi\rangle^{k}\widehat{f}\left(\xi\right)\in L^{2}\left(\IR^{d},H^{r}\right).
	\end{align}
	
	 We abbreviate partial derivatives by
	\begin{align}
		\left(\p^{\gamma}f\right)\left(x\right)=\frac{\p^{\left|\gamma\right|}}{\p_{x^{1}}^{\gamma_{1}}\cdot\ldots\cdot\p_{x^{d}}^{\gamma_{d}}}f\left(x^{1},\ldots,x^{d}\right),\ x\in\IR^{d},\ f\in C^{\infty}\left(\IR^{d}\right),\ \gamma\in\IN^{d},
	\end{align}
	and denote the extension of $\p^{\gamma}$ to the Sobolev space $W^{\left|\gamma\right|,2}\left(\IR^{d},H^{r}\right)$ with the same letter. We note that for $k\in\IN$,
	\begin{align}
		f\in W^{k,2}\left(\IR^{d},H^{r}\right)\Leftrightarrow \p^{\gamma}f\in L^{2}\left(\IR^{d},H^{r}\right),\ \left|\gamma\right|\leq k.
	\end{align}
	
	Let $c$ be a Clifford multiplication over $\IR^{d}$ of rank $r\in\IN$, and consider the associated Dirac operator $c\nabla$ on $\IR^{d}$, i.e. there are Clifford matrices $\left(c\left(\Id x^{j}\right)\right)_{j=1}^{d}\in\IC^{r\times r}$, such that
	\begin{align}
		c\left(\Id x^{k}\right)c\left(\Id x^{l}\right)+c\left(\Id x^{l}\right)c\left(\Id x^{k}\right)=-2\delta_{kl}\one_{\IC^{r\times r}},\ k,l\in\left\{1,\ldots,d\right\}.
	\end{align}
	Then $c\nabla$ is the operator in $L^2\left(\IR^{d},H^{r}\right)$ given by
	\begin{align}
		c\nabla f&:=\sum_{j=1}^{d}c\left(\Id x^{j}\right)\p_{x^{j}}f,\nonumber\\
		\dom\left(c\nabla\right)&:=W^{1,2}\left(\IR^{d},H^{r}\right).
	\end{align}
	It is well-known that $c\nabla$ with the above domain is self-adjoint, and that $\dom\left(\langle c\nabla\rangle^{k}\right)=W^{k,2}\left(\IR^{d},H^{r}\right)$ as Banach spaces. The square of the Dirac operator is the Laplacian $\Delta=\left(c\nabla\right)^{2}$, with
	\begin{align}
		\Delta f&=-\sum_{j=1}^{d}\p_{x_{j}}^{2}f,\nonumber\\
		\dom\left(\Delta\right)&=W^{2,2}\left(\IR^{d},H^{r}\right).
	\end{align}
	Note that $\Delta+1=\langle c\nabla\rangle^{2}$.
	
	The Fourier transform $\mathcal{F}$	diagonalizes $\Delta$, i.e. for $f:\IR\rightarrow\IR$ a Borel function, and if $f\left(\left|X\right|^{2}\right)$ denotes the multiplication operator
	\begin{align}
		\left(f\left(\left|X\right|^{2}\right)g\right)\left(\xi\right)&:=f\left(\left|\xi\right|^{2}\right)g\left(\xi\right),\ \xi\in\IR^{d},\nonumber\\
		\dom\left(f\left(\left|X\right|^{2}\right)\right)&:=\left\{g\in L^2\left(\IR^{d},H^{r},\frac{1}{\left(2\pi\right)^{d}}\Id\xi\right)\right.\nonumber\\
		&\quad\quad\left.\left|\xi\mapsto f\left(\left|\xi\right|^{2}\right)g\left(\xi\right)\in L^2\left(\IR^{d},H^{r},\frac{1}{\left(2\pi\right)^{d}}\Id\xi\right)\right.\right\},
	\end{align}
	we have that
	\begin{align}
		\mathcal{F}:\dom\left(f\left(\Delta\right)\right)\rightarrow\dom\left(f\left(\left|X\right|^{2}\right)\right)
	\end{align}
	is an isometry with respect to the graph norms and
	\begin{align}
		f\left(\Delta\right)=\mathcal{F}^{-1}f\left(\left|X\right|^{2}\right)\mathcal{F}.
	\end{align}
	
	Throughout this paper we assume that $A_{0}$ is a given self-adjoint operator in $H^{r}$ with domain $\dom\left(A_{0}\right)$, which we may consider as a Banach space if equipped with the graph norm of $A_{0}$. We introduce the space
	\begin{align}
		\mathcal{D}:=\bigcup_{n\in\IN}\rg\left(\one_{\left[-n,n\right]}\left(A_{0}\right)\right)\subseteq H^{r},
	\end{align}
	which is dense in $H^{r}$, a simple consequence of self-adjoint functional calculus.
	
	In $L^2\left(\IR^{d},H^{r}\right)$ we may define the constant multiplication operator $\widehat{A_{0}}$ by
	\begin{align}
		\left(\widehat{A_{0}}f\right)\left(x\right)&:=A_{0}f\left(x\right),\ x\in\IR^{d},\nonumber\\
		\dom\left(\widehat{A_{0}}\right)&:=L^2\left(\IR^{d},\dom\left(A_{0}\right)\right).
	\end{align}

	We introduce the constant coefficient Callias operator $D_{0}$ in $L^2\left(\IR^{d},\IC^{r}\right)$, given by
	\begin{align}
		D_{0}&:=ic\nabla+\widehat{A_{0}},\nonumber\\
		\dom\left(D_{0}\right)&=W^{1,2}\left(\IR^{d},H^{r}\right)\cap\dom\left(\widehat{A_{0}}\right).
	\end{align}
	It has been shown in (Lemma 4.2, \cite{GLMST}) that $D_{0}$ is closed and its graph-norm coincides with the sum of the norms of $W^{1,2}\left(\IR^{d},H^{r}\right)$ and $L^2\left(\IR^{d},\dom\left(A_{0}\right)\right)$. Furthermore $D_{0}$ is normal with
	\begin{align}
		D_{0}^{\ast}&=-ic\nabla+\widehat{A_{0}},\nonumber\\
		\dom\left(D_{0}^{\ast}\right)&=\dom\left(D_{0}\right).
	\end{align}
	The automatically non-negative, self-adjoint operator $H_{0}:=D_{0}^{\ast}D_{0}$ satisfies $H_{0}=\Delta+\widehat{A_{0}}^{2}$ via the commuting functional calculi of $c\nabla$ and $\widehat{A_{0}}$.	
	
	Let us now introduce the first basic conditions on the operator family $A$.	
	\begin{Hypothesis}\label{basichyp}
		Assume $A=\left(A\left(x\right)\right)_{x\in\IR^{d}}$ is a family of symmetric operators in $H^{r}$, with $\dom\left(A\left(x\right)\right)\supseteq\mathcal{D}$, $x\in\IR^{d}$. Furthermore we suppose that for all $\phi,\psi\in\mathcal{D}$,
		\begin{align}
			\IR^{d}\ni x\mapsto\langle A\left(x\right)\phi,\psi\rangle_{H^{r}}
		\end{align}
		is (Lebesgue-)measurable.
	\end{Hypothesis}
	
	Under these conditions on $A$ we may define the multiplications operator $A\left(X\right)$ in $L^2\left(\IR^{d},H^{r}\right)$ by
	\begin{align}
		\left(A\left(X\right)f\right)\left(x\right)&:=A\left(x\right)f\left(x\right),\ x\in\IR^{d},\nonumber\\
		\dom\left(A\left(X\right)\right)&:=\left\{f\in L^2\left(\IR^{d},H^{r}\right)\left|x\mapsto A\left(x\right)f\left(x\right)\in L^2\left(\IR^{d},H^{r}\right)\right.\right\}.
	\end{align}
	Note that so far we have not enough assumptions on the family $A$ to claim that $A\left(X\right)$ is densely defined.
	
	However we may associate the Callias operator $D$ to the operator family $A$ via
	\begin{align}
		D&=ic\nabla+A\left(X\right),\nonumber\\
		\dom\left(D\right)&=\dom\left(c\nabla\right)\cap\dom\left(A\left(X\right)\right).
	\end{align}
	Again, we can not claim that $D$ is densely defined. This will be our first goal in the sequel to find additional conditions on the family $A$ to conclude that $D$ is a closed operator with domain $\dom\left(D\right)=\dom\left(D_{0}\right)$.
	
	Before we proceed with the first results, we also need to introduce a notion of differentiability for the family $A$.	
	\begin{Definition}\label{opdiffdef}
		Let $\left(Y,\left\|\cdot\right\|_{Y}\right)$ be a Banach space, and assume that $\mathcal{E}\subseteq Y$ is dense. Let $\left(B\left(x\right)\right)_{x\in\IR^{d}}$ be a measurable family of operators in $Y$, with domains $\dom\left(B\left(x\right)\right)\supseteq\mathcal{E}$, $x\in\IR^{d}$.
		If for $\gamma\in\IN^{d}$ there exists a measurable family of densely defined operators $C_{\gamma}\left(x\right)$, $x\in\IR^{d}$, with $\dom\left(C_{\gamma}\left(x\right)\right)\supseteq\mathcal{E}$, such that for all $f\in C_{c}^{\infty}\left(\IR^{d}\right)\otimes\mathcal{E}$,
		\begin{align}
			\int_{\IR^{d}}C_{\gamma}\left(x\right)f\left(x\right)\Id x=\left(-1\right)^{\left|\beta\right|}\int_{\IR^{d}}B\left(x\right)\p^{\beta}f\left(x\right)\Id x,
		\end{align}
		where the integral converges in $Y$ as Bochner integrals, we say $\left(B\left(x\right)\right)_{x\in\IR^{d}}$ is $\gamma$-times weakly differentiable on $\mathcal{E}$, and we set
		\begin{align}
			\left(\p^{\gamma,End}B\right)\left(x\right)\psi:=C_{\gamma}\left(x\right)\psi,\ \psi\in\mathcal{E}.
		\end{align}
		If additionally for all $\left|\gamma\right|\leq n\in\IN$, $K\subset\IR^{d}$ compact, and $\psi\in\mathcal{E}$,
		\begin{align}
			\int_{K}\left\|\left(\p^{\gamma,End}B\right)\left(x\right)\psi\right\|^{2}_{Y}\Id x<\infty,
		\end{align}
		we write $B\in W^{n,2,End}_{loc}\left(\IR^{d},\left(\mathcal{E},Y\right)\right)$.
	\end{Definition}
	In our setup the spaces are $Y=H^{r}$ and $\mathcal{E}=\mathcal{D}$.
	
	\section{Domain Properties of $A\left(X\right)$ and $D$}
	
		In this section we will establish conditions on the family $A$ which ensure that the Callias operator $D$ is closed, the operators $D^{\ast}D$ and $DD^{\ast}$ are self-adjoint, and their powers have the same domain as the powers of $H_{0}=D_{0}^{\ast}D_{0}$.
		
		We begin by by establishing a class of dense sets in $L^2\left(\IR^{d},H^{r}\right)$, which will be cores for some operators in $L^2\left(\IR^{d},H^{r}\right)$ we will discuss. 
		
		\begin{Lemma}\label{densesetlem}
			Let $s\geq 0$. Then
			\begin{align}
				\rg\left(\left.\left(H_{0}+1\right)^{s}\right|_{C_{c}^{\infty}\left(\IR^{d}\right)\otimes\mathcal{D}}\right)
			\end{align}
			is dense in $L^2\left(\IR^{d},H^{r}\right)$.
		\end{Lemma}
		
		\begin{proof}
			Let $f\in L^2\left(\IR^{d},H^{r}\right)$, such that for all $g\in C_{c}^{\infty}\left(\IR^{d}\right)$ and $\psi\in\mathcal{D}$,
			\begin{align}
				\langle\left(H_{0}+1\right)^{s}\left(g\otimes\psi\right),f\rangle_{L^2\left(\IR^{d},H^{r}\right)}=0.
			\end{align}
			If $\widehat{g}$ denotes the Fourier transform of $g$ and $\widehat{f}\in L^2\left(\IR^{d},H^{r}\right)$ the Fourier transform of $f$, by Plancherel's theorem we find that
			\begin{align}
				\int_{\IR^{d}}\widehat{g}\left(\xi\right)\langle\left(1+\left|\xi\right|^{2}+A_{0}^{2}\right)^{s}\psi,\widehat{f}\left(\xi\right)\rangle_{H^{r}}\Id\xi=0.
			\end{align}
			Because $\left\{\widehat{g}\left|g\in C_{c}^{\infty}\left(\IR^{d}\right)\right.\right\}$ is dense in $L^2\left(\IR^{d}\right)$, we conclude that for all $\psi\in\mathcal{D}$ exists a nullset $N_{\psi}\subset\IR^{d}$, such that for $\xi\in\IR^{d}\backslash N_{\psi}$
			\begin{align}
				\langle\left(1+\left|\xi\right|^{2}+A_{0}^{2}\right)^{s}\psi,\widehat{f}\left(\xi\right)\rangle_{H^{r}}=0.
			\end{align}
			$H^{r}$ is separable and $\mathcal{D}$ a dense subspace. Therefore there exists an at most countable subset $\left\{\psi_{n}\right\}_{n\in\IN}\subseteq\mathcal{D}$, which is dense in $H^{r}$.
			On the other hand for all $\xi\in\IR^{d}$ we have $\phi_{n}:=\left(1+\left|\xi\right|^{2}+A_{0}^{2}\right)^{-s}\psi_{n}\in\mathcal{D}$ for $n\in\IN$. Also $N=\bigcup_{n\in\IN}N_{\phi_{n}}$ is still a nullset in $\IR^{d}$. Thus for $\xi\in\IR^{d}\backslash N$ and all $n\in\IN$,
			\begin{align}
				\langle\psi_{n},\widehat{f}\left(\xi\right)\rangle_{H^{r}}=\langle\left(1+\left|\xi\right|^{2}+A_{0}^{2}\right)^{s}\phi_{n},\widehat{f}\left(\xi\right)\rangle_{H^{r}}=0.
			\end{align}
			But $\left\{\psi_{n}\right\}_{n\in\IN}$ is dense in $H^{r}$, so $\widehat{f}\left(\xi\right)=0$ for a.e. $\xi\in\IR^{d}$. Thus $f=0$ as an element of $L^2\left(\IR^{d},H^{r}\right)$. Consequently $\rg\left(\left.\left(H_{0}+1\right)^{s}\right|_{C_{c}^{\infty}\left(\IR^{d}\right)\otimes\mathcal{D}}\right)$ is dense in $L^2\left(\IR^{d},H^{r}\right)$. 
		\end{proof}
		
		Let us show that the operator families $B$ we want to consider give rise to densely defined multiplication operators $B\left(X\right)$.
	
		\begin{Lemma}\label{densemultiplierdomlem}
			Let
			\begin{align}
				B\in L^{2,End}_{loc}\left(\IR^{d},\left(\mathcal{D},H^{r}\right)\right):=W^{0,2,End}_{loc}\left(\IR^{d},\left(\mathcal{D},H^{r}\right)\right).
			\end{align}
			Then $C_{c}\left(\IR^{d}\right)\otimes\mathcal{D}\in\dom\left(B\left(X\right)\right)$. Especially $B\left(X\right)$ is densely defined.
		\end{Lemma}
		
		\begin{proof}
			Let $g\in C^{\infty}_{c}\left(\IR^{d}\right)$ and $\psi\in\mathcal{D}$.	We need to show that $x\mapsto g\left(x\right)B\left(x\right)\psi\in L^2\left(\IR^{d},H^{r}\right)$, the statement of the Lemma then follows by linearity of $B\left(X\right)$. To this end let $\supp g\subseteq K$, where $K\subset\IR^{d}$ is compact. Then
			\begin{align}
				\int_{\IR^{d}}\left|g\left(x\right)\right|^{2}\left\|B\left(x\right)\psi\right\|_{H^{r}}^{2}\Id x\leq\sup_{x\in K}\left|g\left(x\right)\right|^{2}\cdot\int_{K}\left\|B\left(x\right)\psi\right\|_{H^{r}}^{2}\Id x<\infty.
			\end{align}
		\end{proof}
	
		The next Lemma enables us to decide if a family of bounded operators $B$ gives rise to a bounded multiplication operator $B\left(X\right)$ between $L^{q}$- and $L^{2}$-spaces.
	
		\begin{Lemma}\label{holderdomlem}
			Let $B\in L^{2,End}_{loc}\left(\IR^{d},\left(\mathcal{D},H^{r}\right)\right)$. Assume for $p\geq 2$,
			\begin{align}
				x\mapsto\left\|B\left(x\right)\right\|_{B\left(H^{r}\right)}\in L^{p}\left(\IR^{d}\right),
			\end{align}
			then for $\frac{1}{p}+\frac{1}{q}=\frac{1}{2}$,
			\begin{align}\label{holderdomlemeq1}
				L^{q}\left(\IR^{d},H^{r}\right)&\subseteq\dom\left(B\left(X\right)\right),\nonumber\nonumber\\
				\left\|B\left(X\right)\right\|_{L^{q}\left(\IR^{d},H^{r}\right)\to L^{2}\left(\IR^{d},H^{r}\right)}&\leq\left\|x\mapsto\left\|B\left(x\right)\right\|_{B\left(H^{r}\right)}\right\|_{L^{p}\left(\IR^{d}\right)}.
			\end{align}
		\end{Lemma}
		
		\begin{proof}
			The statement follows immediately by H\"older's inequality.
		\end{proof}
	
		Since we want to use Kato-Rellich's theorem for perturbations of the operator $H_{0}^{\frac{N}{2}}$, we will need to know if a multiplication operator $B\left(X\right)$ is relatively bounded with bound strictly less than $1$. The next Lemma establishes such bounds via Sobolev inequalities.
		
		\begin{Lemma}\label{sobemblem}
			Let $B\in L^{2,End}_{loc}\left(\IR^{d},\left(\mathcal{D},H^{r}\right)\right)$, and $u\geq 0$. Assume there exists $t\in\left[0,u\right]$ and $p\in\left[2,+\infty\right]$, such that
			\begin{align}
				\left\|x\mapsto\left\|B\left(x\right)\langle A_{0}\rangle^{-2t}\right\|_{B\left(H^{r}\right)}\right\|_{L^{p}\left(\IR^{d}\right)}<\infty,\ \text{for }u-t>\frac{d}{2p},
			\end{align}
			or
			\begin{align}
				&\left\|x\mapsto\left\|B\left(x\right)\left(A_{0}^{2}+z\right)^{-t}\right\|_{B\left(H^{r}\right)}\right\|_{L^{p}\left(\IR^{d}\right)}=o\left(1\right),\ z\to+\infty,\nonumber\\
				&\text{for }\left(u=t, p=2\right)\vee\left(d\geq 3,\frac{d}{p}\in\IN,u-t=\frac{d}{2p}\right),
			\end{align}
			then
			\begin{align}
				\dom\left(H_{0}^u\right)&\subseteq\dom\left(B\left(X\right)\right),\nonumber\\
				\left\|B\left(X\right)\left(H_{0}+z\right)^{-u}\right\|_{B\left(L^2\left(\IR^{d},H^{r}\right)\right)}&=o\left(1\right),\ z\to+\infty.
			\end{align}
		\end{Lemma}
		
		\begin{proof}
			Let $z\geq 1$. Then for $p\in\left(2,+\infty\right]$, $\frac{1}{p}+\frac{1}{q}=\frac{1}{2}$,
			$\frac{1}{q}+\frac{1}{q'}=1$, and $f\in L^2\left(\IR^{d},H^{r}\right)$,
			\begin{align}
				&\left\|\left(\Delta+z\right)^{-s}f\right\|_{L^{q}\left(\IR^{d},H^{r}\right)}\leq\left\|\left(\left|X\right|^{2}+z\right)^{-s}\widehat{f}\right\|_{L^{q'}\left(\IR^{d},H^{r}\right)}\nonumber\\
				\leq&\left(\frac{1}{\left(2\pi\right)^{d}}\int_{\IR^{d}}\left(\left|\xi\right|^{2}+z\right)^{-sp}\Id\xi\right)^{\frac{1}{p}}\left\|f\right\|_{L^2\left(\IR^{d},H^{r}\right)}\nonumber\\
				=&\left(2\pi\right)^{-pd}z^{\frac{d}{2p}-s}\left\|\xi\mapsto\langle\xi\rangle^{-2s}\right\|_{L^{p}\left(\IR^{d}\right)}\left\|f\right\|_{L^2\left(\IR^{d},H^{r}\right)}.
			\end{align}
		Consequently for $s>\frac{d}{2p}$ there is a constant $C<\infty$, such that 
		\begin{align}
			\left\|\left(\Delta+z\right)^{-s}\right\|_{L^2\left(\IR^{d},H^{r}\right)\to L^{q}\left(\IR^{d},H^{r}\right)}\leq Cz^{\frac{d}{2t}-s}=o\left(1\right),\ z\to+\infty.
		\end{align}
		For $p=2$ we obtain
		\begin{align}
			\left\|\left(\Delta+z\right)^{-s}\right\|_{B\left(L^2\left(\IR^{d},H^{r}\right)\right)}&=\left\|\left(\left|X\right|^2+z\right)^{-s}\right\|_{B\left(L^2\left(\IR^{d},H^{r},\frac{1}{\left(2\pi\right)^{d}}d\xi\right)\right)}=\esssup_{\xi\in\IR^{d}}\left(\left|\xi\right|^2+z\right)^{-s}\nonumber\\
			&=z^{-s}=\begin{cases}o\left(1\right),\ z\to+\infty,&s>0,\\
				1,&s=0.			
		\end{cases}
		\end{align}
		If $d\geq 3$ and $2s\geq\lceil\frac{d}{p}\rceil$, then the Sobolev inequality (Theorem 4.12 Part III, \cite{Adams}) yields a constant $C<\infty$, such that for $z\geq 1$,
		\begin{align}
			\left\|\left(\Delta+z\right)^{-s}\right\|_{L^2\left(\IR^{d},H^{r}\right)\to L^{q}\left(\IR^{d},H^{r}\right)}\leq C.
		\end{align}		
		This result is only interesting for $\frac{d}{p}\in\IN$, otherwise it is superseded by previous estimates.
		
		We also note that by the commuting functional calculi of $\Delta$ and $\widehat{A_{0}}$, that for $s,t\geq 0$, $s+t=u$,
		 \begin{align}
		 	\left\|\left(\Delta+z\right)^{s}\left(\widehat{A_{0}}^{2}+z\right)^{t}\left(H_{0}+z\right)^{-u}\right\|_{B\left(L^2\left(\IR^{d},H^{r}\right)\right)}\leq 1,\ z\geq 1,
		 \end{align}
		Consequently we obtain for $s,t\geq 0$, $s+t=u$, by Lemma \ref{holderdomlem} and the assumptions in this Lemma,
		\begin{align}
			&\left\|B\left(X\right)\left(H_{0}+z\right)^{-u}\right\|_{B\left(L^2\left(\IR^{d},H^{r}\right)\right)}\nonumber\\
			\leq&\left\|B\left(X\right)\left(\widehat{A_{0}}^{2}+z\right)^{-t}\right\|_{L^{q}\left(\IR^{d},H^{r}\right)\to L^2\left(\IR^{d},H^{r}\right)}\left\|\left(\Delta+z\right)^{-s}\right\|_{L^2\left(\IR^{d},H^{r}\right)\to L^{q}\left(\IR^{d},H^{r}\right)}\nonumber\\
			=&\left\|x\mapsto\left\|B\left(x\right)\left(A_{0}^{2}+z\right)^{-t}\right\|_{B\left(H^{r}\right)}\right\|_{L^{p}\left(\IR^{d}\right)}\nonumber\\
			&\cdot\begin{cases}
				o\left(1\right),&s>\frac{d}{2p},\\
				C,&\left(s=0,p=2\right)\vee\left(d\geq 3,\frac{d}{p}\in\IN,s=\frac{d}{2p}\right),
			\end{cases}\ z\to+\infty,\nonumber\\
		=&o\left(1\right),\ z\to+\infty.
		\end{align}
	\end{proof}

The next ingredient to discuss perturbations of $H_{0}$-powers is a Leibniz rule for multiplication operators $B\left(X\right)$ in the next statement. The proof is standard, we however give it for self sufficiency in this operator-valued setup.

\begin{Lemma}\label{leibnizlem}
	Let $B\in W^{N,2,End}_{loc}\left(\IR^{d},\left(\mathcal{D},H^{r}\right)\right)$. for some $N\in\IN$. If $f\in C_{c}^{\infty}\left(\IR^{d}\right)\otimes\mathcal{D}$, then $B\left(X\right)f\in\dom\left(\p^{\gamma}\right)= W^{\left|\gamma\right|,2}\left(\IR^{d},H^{r}\right)$ for $\gamma\in\IN^{d}$, $\left|\gamma\right|\leq N$, and
	\begin{align}
		\p^{\gamma}B\left(X\right)f=\sum_{\delta\leq\gamma}\binom{\gamma}{\delta}\left(\p^{\delta,End}B\right)\left(X\right)\p^{\gamma-\delta}f.
	\end{align}
\end{Lemma}

\begin{proof}

	We proceed by induction on $N$. The statements of the Lemma hold for $\left|\gamma\right|=0$, by Lemma \ref{densemultiplierdomlem}. Let us assume the statements are true for $n\leq N$. Let $\left|\widehat{\gamma}\right|=N+1$. Then there is $\gamma\in\IN^{d}$, and $j\in\left\{1,\ldots,d\right\}$, such that $\widehat{\gamma}=\gamma+\delta_{j}$. Since by assumption $B\left(X\right) f\in W^{N,2}\left(\IR^{d},H^{r}\right)$, it suffices to check that $\p^{\gamma}B\left(X\right)f\in W^{1,2}\left(\IR^{d},H^{r}\right)$ to conclude $B\left(X\right)f\in W^{N+1,2}\left(\IR^{d},H^{r}\right)$. Also by assumption we know that
	\begin{align}\label{leibnizlemeq0}
		\p^{\gamma}B\left(X\right)f=\sum_{\delta\leq\gamma}\binom{\gamma}{\delta}\left(\p^{\delta,End}B\right)\left(X\right)\p^{\gamma-\delta}f.
	\end{align}
	Since $\left(\p^{\delta,End}B\right)\left(X\right)\p^{\gamma-\delta}f\in L^2\left(\IR^{d},\H^{r}\right)$ by Lemma \ref{densemultiplierdomlem}, we may consider\break$\left(\p^{\delta,End}B\right)\left(X\right)\p^{\gamma-\delta}f$ as a $H^{r}$-valued distribution. As such, for any $\phi\in C_{c}^{\infty}\left(\IR^{d}\right)$,
	\begin{align}\label{leibnizlemeq1}
		&\left(\p_{x_{j}}\left(\p^{\delta,End}B\right)\left(X\right)\p^{\gamma-\delta}f\right)\left[\phi\right]=-\int_{\IR^{d}}\left(\p^{\delta,End}B\right)\left(x\right)\left(\p^{\gamma-\delta}f\right)\left(x\right)\left(\p_{x_{i}}\phi\right)\left(x\right)\Id x\\
		=&\left(\left(\p^{\delta,End}B\right)\left(X\right)\p^{\gamma-\delta+\delta_{j}}f\right)\left[\phi\right]
		-\int_{\IR^{d}}\left(\p^{\delta,End}B\right)\left(x\right)\left(\p_{x_{i}}\left(\left(\p^{\gamma-\delta}f\right)\phi\right)\right)\left(x\right)\Id x\\
		=&\left(\left(\p^{\delta,End}B\right)\left(X\right)\p^{\gamma-\delta+\delta_{j}}f\right)\left[\phi\right]+\left(\left(\p^{\delta+\delta_{j},End}B\right)\left(X\right)\p^{\gamma-\delta}f\right)\left[\phi\right].
	\end{align}
	The above steps are justified by Lemma \ref{densemultiplierdomlem} and the fact that for any $\eta\in\IN^{d}$ the derivative $\p^{\eta}f\in C_{c}^{\infty}\left(\IR^{d}\right)\otimes\mathcal{D}$, and $\p^{\delta,End}B,\ \p^{\delta+\delta_{j},End}B\in L^{2,End}_{loc}\left(\IR^{d},\left(\mathcal{D},H^{r}\right)\right)$. This argument also yields
	\begin{align}
		\left(\p^{\delta,End}B\right)\left(X\right)\p^{\gamma-\delta+\delta_{j}}+\left(\p^{\delta+\delta_{j},End}B\right)\left(X\right)\p^{\gamma-\delta}f\in L^2\left(\IR^{d},H^{r}\right),
	\end{align}
	which implies by (\ref{leibnizlemeq1}) that
	\begin{align}
		\left(\p^{\delta,End}B\right)\left(X\right)\p^{\gamma-\delta}f\in W^{1,2}\left(\IR^{d},H^{r}\right),
	\end{align}
	and thus by (\ref{leibnizlemeq0}), $\p^{\gamma}B\left(X\right)f\in W^{1,2}\left(\IR^{d},H^{r}\right)$. Finally we combine (\ref{leibnizlemeq0}) and (\ref{leibnizlemeq1}) to conclude
	\begin{align}
		\p^{\widehat{\gamma}}B\left(X\right)f&=\sum_{\delta\leq\gamma}\binom{\gamma}{\delta}\left(\left(\p^{\delta+\delta_{j},End}B\right)\left(X\right)\p^{\gamma-\delta}f+\left(\p^{\delta,End}B\right)\left(X\right)\p^{\gamma-\delta+\delta_{j}}f\right)\\
		&=\sum_{\beta\leq\widehat{\gamma}}\binom{\widehat{\gamma}}{\delta}\left(\p^{\delta,End}B\right)\left(X\right)\p^{\widehat{\gamma}-\delta}f.
	\end{align}
\end{proof}

The next Lemma is closely related with the Leibniz rule in Lemma \ref{leibnizlem}, however its usefulness will become only apparent in the next section, when we discuss trace-class properties of a certain class of operators in $L^2\left(\IR^{d},H^{r}\right)$. Its rough content is that differentiability of an operator family $B$ enables us to extract smoothing operators in $L^2\left(\IR^{d},H^{r}\right)$ from the left of the multiplication operator $B\left(X\right)$.

\begin{Lemma}\label{leibnizsplitlem}
	Let $B\in W^{2l+k,2,End}_{loc}\left(\IR^{d},\left(\mathcal{D},H^{r}\right)\right)$ for $l\in\IN$ and $k\in\left\{0,1\right\}$. Then there are constant matrices
	\begin{align}
		C^{k,l}_{\gamma,\delta}\in\left\{\lambda_{0}1_{H^{r}}+\sum_{j=1}^{d}\lambda_{j}c\left(\Id x^{j}\right),\ \lambda_{j}\in\IZ+i\IZ,\ j\in\left\{0,\ldots,d\right\}\right\},\ \gamma,\delta\in\IN^{d},
	\end{align}
	such that for $f\in C_{c}^{\infty}\left(\IR^{d}\right)\otimes\mathcal{D}$,
	\begin{align}
		B\left(X\right)f=\left(ic\nabla+1\right)^{-k}\left(\Delta+1\right)^{-l}\sum_{\stackrel{\gamma,\delta\in\IN^{d}}{\left|\gamma+\delta\right|\leq 2l+k}}C_{\gamma,\delta}^{k,l}\left(\p^{\gamma,End}B\right)\left(X\right)\p^{\delta}f,
	\end{align}
\end{Lemma}

\begin{proof}
	The conditions on $B$ imply by Lemma \ref{leibnizlem} that for any $\gamma\in\IN^{d}$, $\left|\gamma\right|\leq 2l+k$,
	\begin{align}
		B\left(X\right)f\in\dom\left(\p^{\gamma}\right).
	\end{align}
	On the other hand we may expand $\left(c\nabla+i\right)^{k}\left(\Delta+1\right)^{l}$ on $C_{c}^{\infty}\left(\IR^{d}\right)\otimes\mathcal{D}$ to find matrices
	\begin{align}
		c^{k,l}_{\gamma,\delta}\in\left\{\lambda_{0}1_{H^{r}}+\sum_{j=1}^{d}\lambda_{j}c\left(\Id x^{j}\right),\ \lambda_{j}\in\IZ+i\IZ,\ j\in\left\{0,\ldots,d\right\}\right\},
	\end{align}
	such that
	\begin{align}
		\left(c\nabla+i\right)^{k}\left(\Delta+1\right)^{l}g=\sum_{\stackrel{\alpha\in\IN^{d}}{\left|\alpha\right|\leq 2l+k}}c_{\alpha}^{k,l}\p^{\alpha}g,\ g\in C_{c}^{\infty}\left(\IR^{d}\right)\otimes\mathcal{D}.
	\end{align}
	Thus $B\left(X\right)f\in\dom\left(\left(c\nabla+i\right)^{k}\left(\Delta+1\right)^{l}\right)$, and by Lemma \ref{leibnizlem},
	\begin{align}
		\left(\left(c\nabla+i\right)^{k}\left(\Delta+1\right)^{l}\right)B\left(X\right)f=\sum_{\stackrel{\gamma\in\IN^{d}}{\left|\gamma\right|\leq 2l+k}}c_{\alpha}^{k,l}\sum_{\delta\leq\gamma}\binom{\gamma}{\delta}\left(\p^{\delta,End}B\right)\left(X\right)\p^{\gamma-\delta}f
	\end{align}
	By applying $\left(c\nabla+i\right)^{-k}\left(\Delta+1\right)^{-l}$ to the left, and using that the resolvents of $c\nabla$ and $\Delta$ commute, we obtain the claimed statement with $C^{k,l}_{\gamma,\delta}=\binom{\gamma+\delta}{\gamma}c_{\gamma}^{k,l}$.
\end{proof}
%
%
	
	We arrive at the fundamental result of this section which establishes conditions on the operator family $A$ such that the Callias operator $D$ is closed and $\dom\left(\left(D^{\ast}D\right)^{\frac{N}{2}}\right)=\dom\left(\left(DD^{\ast}\right)^{\frac{N}{2}}\right)=\dom\left(H_{0}^{\frac{N}{2}}\right)$, for given $N\in\IN$.
	
		\begin{Proposition}\label{hdomainprop}
		Let $N\in\IN$, $N\geq 1$, and $A\in W^{N-1,2,End}_{loc}\left(\IR^{d},\left(\mathcal{D},H^{r}\right)\right)$. Assume for $\gamma\in\IN^{d}$, $1\leq\left|\gamma\right|\leq N-1$, there exist $t\in\left[0,\frac{\left|\gamma\right|}{2}\right]$ and $p\in\left[2,+\infty\right]$ such that
		\begin{align}
			\left\|x\mapsto\left\|\left(\p^{\gamma,End}A\right)\left(x\right)\langle A_{0}\rangle^{-2t}\right\|_{B\left(H^{r}\right)}\right\|_{L^{p}\left(\IR^{d}\right)}<\infty,\ \text{for }\frac{\left|\gamma\right|}{2}-t>\frac{d}{2p},
		\end{align}
		or
		\begin{align}
			&\left\|x\mapsto\left\|\left(\p^{\gamma,End}A\right)\left(x\right)\left(A_{0}^{2}+z\right)^{-t}\right\|_{B\left(H^{r}\right)}\right\|_{L^{p}\left(\IR^{d}\right)}=o\left(1\right),\ z\to+\infty,\nonumber\\
			&\text{for }\left(\frac{\left|\gamma\right|}{2}=t, p=2\right)\vee\left(d\geq 3,\frac{d}{p}\in\IN,\frac{\left|\gamma\right|}{2}-t=\frac{d}{2p}\right),
		\end{align}
		As well as
		\begin{align}\label{hdomainpropeq0}
			\esssup_{x\in\IR^{d}}\left\|A_{0}^{n}\left(A\left(x\right)-A_{0}\right)\left(A_{0}^{2}+z\right)^{-\frac{n+1}{2}}\right\|_{B\left(H^{r}\right)}=o\left(1\right),\ z\to+\infty,\ 0\leq n\leq N-1.
		\end{align}
		Denote for $j\in\left\{0,1\right\}$ the operators $D_{j}=\left(-1\right)^{j}ic\nabla+A\left(X\right)$, and $D_{j,0}=\left(-1\right)^{j}ic\nabla+\widehat{A_{0}}$. Then for any $\eta\in\left\{0,1\right\}^{N}$, the operator $\prod_{k=1}^{N}D_{\eta_{k}}$ is closed with		
		\begin{align}
			\dom\left(\prod_{k=1}^{N}D_{\eta_{k}}\right)&=\dom\left(\prod_{k=1}^{N}D_{\eta_{k},0}\right)=\dom\left(H_{0}^{\frac{N}{2}}\right)\nonumber\\
			&=W^{N,2}\left(\IR^{d},H^{r}\right)\cap  L^2\left(\IR^{d},\dom\left(A_{0}^{N}\right)\right).
		\end{align}
		Additionally $\prod_{k=1}^{N}D_{\eta_{k}}$ is self-adjoint if it is symmetric.
	\end{Proposition}
	
	\begin{proof}
		We first note that by Lemma \ref{sobemblem} we have $\dom\left(\left(\p^{\gamma,End}A\right)\left(X\right)\right)\supseteq\dom\left(H_{0}^{\frac{\left|\gamma\right|+1}{2}}\right)$ for $\gamma\in\IN^{d}$, $1\leq\left|\gamma\right|\leq N-1$, and
		\begin{align}
			\left\|\left(\p^{\gamma,End}A\right)\left(X\right)\left(H_{0}+z\right)^{-\frac{\left|\gamma\right|+1}{2}}\right\|_{B\left(L^2\left(\IR^{d},H^{r}\right)\right)}=o\left(1\right),\ z\to+\infty.
		\end{align}
	
		Let $n\in\IN$ and let $\phi\in C_{c}^{\infty}\left(\IR^{d}\right)\otimes\mathcal{D}$. Then, if $\widehat{\phi}$ denotes the $H^{r}$-valued Fourier transform of $\phi$,
		\begin{align}\label{hdomainpropeq1}
			\left\|H_{0}^{\frac{n}{2}}\phi\right\|_{L^2\left(\IR^{d},H^{r}\right)}^{2}=\left(2\pi\right)^{-d}\int_{\IR^{d}}\int_{\sigma\left(A_{0}\right)}\left(\left|\xi\right|^2+\lambda^2\right)^{n}\langle \Id E_{A_{0}}\left(\lambda\right)\widehat{\phi}\left(\xi\right),\widehat{\phi}\left(\xi\right)\rangle_{H^{r}}\Id\xi.
		\end{align}
		Then after multiplying out the polynomial $\left(\left|\xi\right|^{2}+\lambda^{2}\right)^{n}$, by Young's inequality, there are constants $C_{n}$ such that
		\begin{align}
			\left|\xi\right|^{2n}+\lambda^{2n}\leq\left(\left|\xi\right|^{2}+\lambda^{2}\right)^{n}\leq C_{n}\left(\left|\xi\right|^{2n}+\lambda^{2n}\right),\ \xi\in\IR^{d},\ \lambda\in\IR.
		\end{align}
		Consequently, by (\ref{hdomainpropeq1}), we have
		\begin{align}\label{h0domaineq}
			\dom\left(H_{0}^{\frac{n}{2}}\right)=\dom\left(\langle c\nabla\rangle^{n}\right)\cap\dom\left(\widehat{A_{0}}^{n}\right)=W^{n,2}\left(\IR^{d},H^{r}\right)\cap  L^2\left(\IR^{d},\dom\left(A_{0}^{n}\right)\right),
		\end{align}
		where the equality also holds with respect to the topologies induced by graph norms. A similar argument, using the commuting Fourier transform and spectral resolution of $A_{0}$, shows that
		\begin{align}
			\left\|\left(\widehat{A_{0}}^2+z\right)^{\frac{n}{2}}\phi\right\|_{L^2\left(\IR^{d},H^{r}\right)}\leq\left\|\left(H_{0}+z\right)^{\frac{n}{2}}\phi\right\|_{L^2\left(\IR^{d},H^{r}\right)},\ z\geq 1.
		\end{align}
		With these preparations we find for $0\leq n\leq N-1$
		\begin{align}\label{hdomainpropeq2}
			&\left\|\widehat{A_{0}}^{n}\left(A\left(X\right)-\widehat{A_{0}}\right)\phi\right\|\nonumber\\
			\leq&\left\|\widehat{A_{0}}^{n}\left(A\left(X\right)-\widehat{A_{0}}\right)\left(\widehat{A_{0}}^{2}+z\right)^{-\frac{n+1}{2}}\right\|_{B\left(L^2\left(\IR^{d},H^{r}\right)\right)}\left\|\left(H_{0}+z\right)^{-\frac{n+1}{2}}\phi\right\|_{L^2\left(\IR^{d},H^{r}\right)}\nonumber\\
			=&\esssup_{x\in\IR^{d}}\left\|A_{0}^{n}\left(A\left(x\right)-A_{0}\right)\left(A_{0}^{2}+z\right)^{-\frac{n+1}{2}}\right\|_{B\left(H^{r}\right)}\left\|\left(H_{0}+z\right)^{-\frac{n+1}{2}}\phi\right\|_{L^2\left(\IR^{d},H^{r}\right)}.
		\end{align}	
		Furthermore there are constant matrices $C^{\gamma,\delta}$ in $\IC^{r\times r}$, such that
		\begin{align}
			&\left(c\nabla\right)^{n}\left(A\left(X\right)-\widehat{A_{0}}\right)\phi=\sum_{\left|\gamma+\delta\right|\leq n}C^{\gamma,\delta}\left(\p^{\gamma, End}\left(A-A_{0}\right)\right)\left(X\right)\p^{\delta}\phi\nonumber\\
			=&\sum_{\left|\gamma+\delta\right|\leq n}C^{\gamma,\delta}\left(\p^{\gamma, End}\left(A-A_{0}\right)\right)\left(X\right)\left(H_{0}+z\right)^{-\frac{\left|\gamma\right|+1}{2}}\p^{\delta}\left(H_{0}+z\right)^{-\frac{\left|\delta\right|}{2}}\left(H_{0}+z\right)^{\frac{\left|\gamma+\delta\right|+1}{2}}\phi.
		\end{align}
		We thus find a constant $C<\infty$, such that
		\begin{align}\label{hdomainpropeq3}
			&\left\|\left(c\nabla\right)^{n}\left(A\left(X\right)-\widehat{A_{0}}\right)\phi\right\|_{L^2\left(\IR^{d},H^{r}\right)}\nonumber\\
			\leq&C\sum_{1\leq\left|\gamma\right|\leq n}\left\|\left(\p^{\gamma, End}\left(A-A_{0}\right)\right)\left(X\right)\left(H_{0}+z\right)^{-\frac{\left|\gamma\right|+1}{2}}\right\|_{B\left(L^2\left(\IR^{d},H^{r}\right)\right)}\left\|\left(H_{0}+z\right)^{\frac{n+1}{2}}\phi\right\|_{L^2\left(\IR^{d},H^{r}\right)}\nonumber\\
			&+\esssup_{x\in\IR^{d}}\left\|\left(A\left(x\right)-\widehat{A_{0}}\right)\left(A_{0}^{2}+z\right)^{-\frac{1}{2}}\right\|_{B\left(H^{r}\right)}\left\|\left(H_{0}+z\right)^{\frac{1}{2}}\phi\right\|_{L^2\left(\IR^{d},H^{r}\right)}.
		\end{align}
		We combine (\ref{h0domaineq}), (\ref{hdomainpropeq2}), (\ref{hdomainpropeq3}), and the prerequisites of this proposition to state
		\begin{align}\label{hdomainpropeq4}
			\left\|H_{0}^{\frac{n}{2}}\left(A\left(X\right)-\widehat{A_{0}}\right)\left(H_{0}+z\right)^{-\frac{n+1}{2}}\right\|_{B\left(L^2\left(\IR^{d},H^{r}\right)\right)}=o\left(1\right),\ z\to+\infty,\ 0\leq n\leq N-1,
		\end{align}
		and since $\left\|\left(H_{0}+z\right)^{s}\left(H_{0}+1\right)^{-s}\right\|_{B\left(L^2\left(\IR^{d},H^{r}\right)\right)}\leq 1$, $z\geq 1$, $s\geq 0$, wich follows by the functional calculi of $\widehat{A_{0}}$ and $\Delta$, we may conclude
		\begin{align}\label{hdomainpropeq4_1}
			\left\|\left(H_{0}+z\right)^{\frac{n}{2}}\left(A\left(X\right)-\widehat{A_{0}}\right)\left(H_{0}+z\right)^{-\frac{n+1}{2}}\right\|_{B\left(L^2\left(\IR^{d},H^{r}\right)\right)}=o\left(1\right),\ z\to+\infty,\ 0\leq n\leq N-1,
		\end{align}
		We want to apply the well-known Kato-Rellich theorems (Theorem 4.1.1, Theorem 5.4.3, \cite{Kato}) to the operator $T_{1}:=\prod_{k=1}^{N}D_{\eta_{k}}$ as a perturbation of $T_{0}=\prod_{k=1}^{N}D_{\eta_{k},0}$. If we show that $T_{1}-T_{0}$ is relatively bounded by $T_{0}$ with a bound strictly less than $1$, then the domains of $T_{0}$ and $T_{1}$ must coincide, $T_{1}$ is closed and additionally self-adjoint if it is symmetric. The spectral resolution of $\widehat{A_{0}}$ and the Fourier transform commute, so we have that $\dom\left(T_{0}\right)=\dom\left(H_{0}^{\frac{N}{2}}\right)$, which implies that it suffices to show the relative bound for $H_{0}^{\frac{N}{2}}$ instead of $T_{0}$. Thus we need to show that
		\begin{align}\label{hdomainpropeq5}
			\sup_{\phi\in C_{c}^{\infty}\left(\IR^{d}\right)\otimes\mathcal{D},\phi\neq 0}\frac{\left\|\left(T_{1}-T_{0}\right)\phi\right\|_{L^2\left(\IR^{d},H^{r}\right)}}{\left\|\left(H_{0}+z\right)^{\frac{N}{2}}\phi\right\|_{L^2\left(\IR^{d},H^{r}\right)}}=o\left(1\right),\ z\to+\infty.
		\end{align}
		We note that for $\phi\in C_{c}^{\infty}\left(\IR^{d}\right)\otimes\mathcal{D}$,
		\begin{align}
			\left(T_{1}-T_{0}\right)\phi=\sum_{\alpha\in\left\{0,1\right\}^{N},\left|\alpha\right|\geq 1}\prod_{k=1}^{N}S_{k,\alpha_{k}}\phi,
		\end{align}
		where $S_{k,0}:=D_{\eta_{k},0}$, and $S_{k,1}=A\left(X\right)-\widehat{A_{0}}$. We treat each summand separately.
		\begin{align}
			\left\|\prod_{k=1}^{N}S_{k,\alpha_{k}}\phi\right\|_{L^2\left(\IR^{d},H^{r}\right)}\leq&\prod_{k=1}^{N}\left\|\left(H_{0}+z\right)^{\frac{k-1}{2}}S_{k,\alpha_{k}}\left(H_{0}+z\right)^{-\frac{k}{2}}\right\|_{B\left(L^2\left(\IR^{d},H^{r}\right)\right)}\nonumber\\
			&\left\|\left(H_{0}+z\right)^{\frac{N}{2}}\phi\right\|_{L^2\left(\IR^{d},H^{r}\right)}.
		\end{align}
		Because $\left\|\left(H_{0}+z\right)^{\frac{k-1}{2}}D_{\eta_{k},0}\left(H_{0}+z\right)^{-\frac{k}{2}}\right\|_{B\left(L^2\left(\IR^{d},H^{r}\right)\right)}\leq 1$ for $z\geq 1$, and at least one of the factors $S_{k,\alpha_{k}}$ equals $A\left(X\right)-\widehat{A_{0}}$, we conclude by (\ref{hdomainpropeq4_1}),
		\begin{align}
			\left\|\prod_{k=1}^{N}S_{k,\alpha_{k}}\phi\right\|_{L^2\left(\IR^{d},H^{r}\right)}=\left\|\left(H_{0}+z\right)^{\frac{N}{2}}\phi\right\|_{L^2\left(\IR^{d},H^{r}\right)}o\left(1\right),\ z\to+\infty,
		\end{align}
		which implies the required asymptotic (\ref{hdomainpropeq5}).
	\end{proof}
	
	\begin{Remark}\label{adomsamerem}
		We note that condition (\ref{hdomainpropeq0}) from Proposition \ref{hdomainprop} implies for a.e. $x\in\IR^{d}$,
		\begin{align}
			\left\|\left(A\left(x\right)-A_{0}\right)\left(A_{0}^{2}+z\right)^{-\frac{1}{2}}\right\|_{B\left(H^{r}\right)}=o\left(1\right),\ z\to+\infty,
		\end{align}
		which implies by the self-adjoint Kato Rellich theorem (Theorem 5.4.3, \cite{Kato}) that $\dom\left(A\left(x\right)\right)=\dom\left(A_{0}\right)$, for a.e. $x\in\IR^{d}$.
	\end{Remark}
	
	\section{Main Results}
	
	In this chapter we will discuss trace class properties of operators arising from multiplication operators, which then leads to the principal result of this work, Theorem \ref{mainthm}. We begin therefore with a Lemma giving conditions on a operator family $B$, such that a class of operators associated with the multiplication operator $B\left(X\right)$ and the Dirac operator $c\nabla$ are Hilbert-Schmidt operators in $L^2\left(\IR^{d},H^{r}\right)$.
	
	\begin{Lemma}\label{traceliftlem}
		Let $\left(B\left(x\right)\right)_{x\in\IR^{d}}$ be a measurable family of operators with $\mathcal{D}\subseteq\dom\left(B\left(x\right)\right)$, $x\in\IR^{d}$, such that
		\begin{align}
			\int_{\IR^{d}}\left\|B\left(x\right)\right\|_{S^1\left(H^{r}\right)}\Id x<\infty.
		\end{align}
		Let $s\geq 0$, $u\geq\frac{s}{2}$, $t<2u-s-\frac{d}{2}$, and $f:\IR^{d}\rightarrow\IC$ measurable with $\esssup_{\xi\in \IR^{d}}\left|f\left(\xi\right)\right|\langle\xi\rangle^{-t}<\infty$,
		then the operator,
		\begin{align}
			Q:=\langle\widehat{A_{0}}\rangle^{s}f\left(c\nabla\right)\left(H_{0}+1\right)^{-u}\left|B\left(X\right)\right|^{\frac{1}{2}},
		\end{align}
		is densely defined and admits a Hilbert-Schmidt extension in $L^2\left(\IR^{d},H^{r}\right)$.
	\end{Lemma}
	
	\begin{proof}
		Let $K\subset\IR^{d}$ be compact and $\psi\in\mathcal{D}$. Then
		\begin{align}
			\int_{K}\left\|\left|B\left(x\right)\right|^{\frac{1}{2}}\psi\right\|_{H^{r}}^{2}\Id x\leq\left\|\psi\right\|_{H^{r}}^{2}\int_{\IR^{d}}\left\|\left|B\left(x\right)\right|^{\frac{1}{2}}\right\|_{S^2\left(H^{r}\right)}^{2}\Id x=\left\|\psi\right\|_{\H^{r}}^{2}\int_{\IR^{d}}\left\|B\left(x\right)\right\|_{S^1\left(H^{r}\right)}\Id x<\infty.
		\end{align}
		So $\left|B\right|^{\frac{1}{2}}\in L^{2,End}_{loc}\left(\IR^{d},\left(\mathcal{D},H^{r}\right)\right)$. Lemma \ref{densemultiplierdomlem} then implies $C_{c}\left(\IR^{d}\right)\otimes\mathcal{D}\subseteq\dom\left(\left|B\left(X\right)\right|^{\frac{1}{2}}\right)$, thus $Q$ is densely defined.\\
		For $\epsilon\searrow0$ and $n\to\infty$ the operators
		\begin{align}
			\kappa_{\epsilon,n}:=e^{-\epsilon\Delta}\one_{\left[-n,n\right]}\left(\widehat{A_{0}}\right),
		\end{align}
		jointly converge strongly to $1$ in $L^2\left(\IR^{d},H^{r}\right)$. Moreover for $\epsilon>0$ and $n\in\IN$
		\begin{align}
			\kappa_{\epsilon,n}\left(x,y\right):=\left(2\pi\right)^{-d}\int_{\IR^{d}}e^{-i\langle x-y,\xi\rangle}e^{-\epsilon\left|\xi\right|^{2}}\Id\xi\ \one_{\left[-n,n\right]}\left(A_{0}\right)
		\end{align}
		is a smooth $B\left(H^{r}\right)$-valued kernel of $\kappa_{\epsilon,n}$. Let $Q_{\epsilon,n}:=\kappa_{\epsilon,n}Q$. For $\epsilon>0$ and $n\in\IN$ its $B\left(H^{r}\right)$-valued integral kernel $k_{Q_{\epsilon,n}}$ at $x$, $y\in\IR^{d}$ is given by
		\begin{align}
			&k_{Q_{\epsilon,n}}\left(x,y\right)\nonumber\\
			=&\left(2\pi\right)^{-d}\int_{\IR^{d}}e^{-i\langle x-y,\xi\rangle}e^{-\epsilon\left|\xi\right|^{2}}f\left(\xi\right)\left(1+\left|\xi\right|^{2}+A_{0}^{2}\right)^{-u}\Id\xi\ \langle A_{0}\rangle^{s}\one_{\left[-n,n\right]}\left(A_{0}\right)\left|B\left(y\right)\right|^{\frac{1}{2}}\nonumber\\
			=&\left(2\pi\right)^{-d}\int_{\IR}\int_{\IR^{d}}e^{-i\langle x-y,\xi\rangle}e^{-\epsilon\left|\xi\right|^{2}}f\left(\xi\right)\left(1+\left|\xi\right|^{2}+\lambda^{2}\right)^{-u}\Id\xi\ \langle \lambda\rangle^{s}\one_{\left[-n,n\right]}\left(\lambda\right)\Id E_{A_{0}}\left(\lambda\right)\ \left|B\left(y\right)\right|^{\frac{1}{2}}
		\end{align}
		Let us introduce the finite Borel measure $\mu_{B}$ on $\IR$ given by
		\begin{align}
			\mu_{B}\left(I\right):=\int_{\IR^{d}}\left\|\one_{I}\left(A_{0}\right)\left|B\left(y\right)\right|^{\frac{1}{2}}\right\|_{S^2\left(H^{r}\right)}^{2}\Id y,\ I\in\mathcal{B}\left(\IR\right),
		\end{align}
		with
		\begin{align}
			\mu_{B}\left(\IR\right)=\int_{\IR^{d}}\left\|\left|B\left(y\right)\right|^{\frac{1}{2}}\right\|_{S^2\left(H^{r}\right)}^{2}\Id y=\int_{\IR^{d}}\left\|B\left(y\right)\right\|_{S^1\left(H^{r}\right)}\Id y<\infty.
		\end{align}
		If $g:\IR\rightarrow\IC$ is an essentially bounded Borel function, then for any orthonormal basis\footnote{$\mathcal{D}$ is a dense linear subspace of a separable Hilbert space, thus such an orthonormal basis always exists.} $\left(\phi_{n}\right)_{n\in\IN}\subseteq\mathcal{D}$ of $H^{r}$,
		\begin{align}\label{traceliftlemeq1}
			&\int_{\IR^{d}}\left\|\int_{\IR}g\left(\lambda\right)\Id E_{A_{0}}\left(\lambda\right)\left|B\left(y\right)\right|^{\frac{1}{2}}\right\|_{S^2\left(H^{r}\right)}^{2}\Id y\nonumber\\
			=&\int_{\IR^{d}}\sum_{n\in\IN}\int_{\IR}\left|g\left(\lambda\right)\right|^{2}\Id\langle E_{A_{0}}\left(\lambda\right)\left|B\left(y\right)\right|^{\frac{1}{2}}\phi_{n},\left|B\left(y\right)\right|^{\frac{1}{2}}\phi_{n}\rangle_{H^{r}}\Id y\nonumber\\
			=&\int_{\IR}\left|g\left(\lambda\right)\right|^{2}\Id\left(\int_{\IR^{d}}\sum_{n\in\IN}\langle E_{A_{0}}\left(\lambda\right)\left|B\left(y\right)\right|^{\frac{1}{2}}\phi_{n},\left|B\left(y\right)\right|^{\frac{1}{2}}\phi_{n}\rangle_{H^{r}}\Id y\right)\nonumber\\
			=&\int_{\IR}\left|g\left(\lambda\right)\right|^{2}\Id\mu_{B}\left(\lambda\right).
		\end{align}
		We return to the integral kernel of $Q_{\epsilon,n}$. For $\epsilon,\delta>0$, $n,m\in\IN$, the function
		\begin{align}
			f_{n,m,\epsilon,\delta}&:\IR^{d}\times\IR\rightarrow\IC\nonumber\\
			f_{n,m,\epsilon,\delta}\left(z,\lambda\right)&:=\left(2\pi\right)^{-d}\langle\lambda\rangle^{s}\int_{\IR^{d}}e^{-i\langle z,\xi\rangle}\left(e^{-\epsilon\left|\xi\right|^{2}}\one_{\left[-n,n\right]}\left(\lambda\right)-e^{-\delta\left|\eta\right|^{2}}\one_{\left[-m,m\right]}\left(\lambda\right)\right)\nonumber\\
			&\quad\quad\quad\quad\quad\quad\quad\quad f\left(\xi\right)\left(1+\left|\xi\right|^{2}+\lambda^{2}\right)^{-u}\Id\xi,
		\end{align}
		is a bounded Borel function in $\lambda$ for every $z\in\IR^{d}$. Additionally $f_{n,m,\epsilon,\delta}$ is given as a partial Fourier transform in $\xi$, thus we may utilize Plancherel's theorem for the $L^2$-norm in $z$. Let us consider the operator $Q_{\epsilon,n}-Q_{\delta,m}$
		\begin{align}\label{traceliftlemeq2}
			&\left\|Q_{\epsilon,n}-Q_{\delta,m}\right\|_{S^2\left(L^2\left(\IR^{d},H^{r}\right)\right)}^{2}=\int_{\IR^{d}}\int_{\IR{d}}\left\|k_{Q_{\epsilon,n}}\left(x,y\right)-k_{Q_{\delta,m}}\left(x,y\right)\right\|_{S^2\left(H^{r}\right)}^{2}\Id x\ \Id y\nonumber\\
			=&\int_{\IR^{d}}\int_{\IR^{d}}\left\|\int_{\IR}f_{n,m,\epsilon,\delta}\left(\left(x-y\right),\lambda\right)\Id E_{A_{0}}\left(\lambda\right)\left|B\left(y\right)\right|^{\frac{1}{2}}\right\|_{S^2\left(H^{r}\right)}^{2}\Id x\ \Id y\nonumber\\
			=&\int_{\IR^{d}}\int_{\IR^{d}}\left\|\int_{\IR}f_{n,m,\epsilon,\delta}\left(z,\lambda\right)\Id E_{A_{0}}\left(\lambda\right)\left|B\left(y\right)^{\ast}\right|^{\frac{1}{2}}\right\|_{S^2\left(H^{r}\right)}^{2}\Id y\ \Id z\nonumber\\
			\stackrel{(\ref{traceliftlemeq1})}{=}&\int_{\IR^{d}}\int_{\IR}\left|f_{n,m,\epsilon,\delta}\left(z,\lambda\right)\right|^{2}\Id\mu_{B}\left(\lambda\right)\ \Id z,
		\end{align}
		We continue by applying Plancherel's theorem,
	\begin{align}
			=&\left(2\pi\right)^{-d}\int_{\IR}\langle\lambda\rangle^{2s}\int_{\IR^{d}}\left|e^{-\epsilon\left|\xi\right|^{2}}\one_{\left[-n,n\right]}\left(\lambda\right)-e^{-\delta\left|\xi\right|^{2}}\one_{\left[-m,m\right]}\left(\lambda\right)\right|^{2}\nonumber\\
			&\quad\quad\quad\quad\quad\quad\quad\quad\left|f\left(\xi\right)\right|^{2}\left(1+\left|\xi\right|^{2}+\lambda^{2}\right)^{-2u}\Id\xi\ \Id\mu_{B}\left(\lambda\right)\nonumber\\
			\leq&\left(2\pi\right)^{-d}\esssup_{\eta\in \IR^{d}}\left|f\left(\eta\right)\right|^{2}\langle\eta\rangle^{-2t}\nonumber\\
			&\int_{\IR}\int_{\IR^{d}}\left|e^{-\epsilon\left|\xi\right|^{2}}\one_{\left[-n,n\right]}\left(\lambda\right)-e^{-\delta\left|\xi\right|^{2}}\one_{\left[-m,m\right]}\left(\lambda\right)\right|^{2}\langle\xi\rangle^{-4u+2s+2t}\Id\xi\ \Id\mu_{B}\left(\lambda\right).
		\end{align}
		The integrand in the last line of (\ref{traceliftlemeq2}) possesses the dominant $\left(\xi,\lambda\right)\mapsto\langle\xi\rangle^{-4u+2s+2t}$, which is integrable against $\Id\xi\otimes\Id\mu_{B}$, because $-4u+2s+2t<-d$ and $\mu_{B}$ is finite. Therefore we may interchange limits in $\epsilon,\delta\searrow 0$ and $n,m\to\infty$ with both integrals. Thus
		\begin{align}
			\left\|Q_{\epsilon,n}-Q_{\delta,m}\right\|_{S^2\left(L^2\left(\IR^{d},H^{r}\right)\right)}^{2}\xrightarrow{\epsilon,\delta\searrow 0,\ n,m\to\infty}0.
		\end{align}
		Since for $\epsilon>0$ and $n\in\IN$ we may similarly show that
		\begin{align}
			\left\|Q_{\epsilon,n}\right\|_{S^2\left(L^2\left(\IR^{d},H^{r}\right)\right)}^{2}\leq&\left(2\pi\right)^{-d}\esssup_{\eta\in \IR^{d}}\left|f\left(\eta\right)\right|^{2}\langle\eta\rangle^{-2t}\nonumber\\
			&\int_{\IR}\int_{\IR^{d}}\left|e^{-\epsilon\left|\xi\right|^{2}}\one_{\left[-n,n\right]}\left(\lambda\right)\right|^{2}\langle\xi\rangle^{-4u+2s+2t}\Id\xi\ \Id\mu_{B}\left(\lambda\right)\nonumber\\
			\leq&\left(2\pi\right)^{-d}\esssup_{\eta\in \IR^{d}}\left|f\left(\eta\right)\right|^{2}\langle\eta\rangle^{-2t}\int_{\IR}\int_{\IR^{d}}\langle\xi\rangle^{-4u+2s+2t}\Id\xi\ \Id\mu_{B}\left(\lambda\right)<\infty,
		\end{align}
		we conclude that the operators $Q_{\epsilon,n}$ admit Hilbert-Schmidt extensions which form a Cauchy sequence of Hilbert-Schmidt operators, and thus converge to some $\widetilde{Q}\in S^2\left(L^2\left(\IR^{d},H^{r}\right)\right)$. On the other hand the strong limit of $Q_{\epsilon,n}$ on their common domain containing $C_{c}^{\infty}\left(\IR^{d}\right)\otimes\mathcal{D}$ is $Q$. Thus $Q$ admits the extension $\widetilde{Q}$.
	\end{proof}

	An immediate consequence of the previous Lemma is that we can give conditions on a operator family $B$, such that a class of operators associated with the multiplication operator $B\left(X\right)$ and the Dirac operator $c\nabla$ become are trace-class operators in $L^2\left(\IR^{d},H^{r}\right)$, by splitting into a product of Hilbert-Schmidt operators.
	
	\begin{Corollary}\label{traceliftcor}
		Let $B\in L^{2,End}_{loc}\left(\IR^{d},\left(\mathcal{D},H^{r}\right)\right)$, such that
		\begin{align}
			\int_{\IR^{d}}\left\|B\left(x\right)\right\|_{S^1\left(H^{r}\right)}\Id x<\infty.
		\end{align}
		Let $s_{j}\geq 0$, $u_{j}\geq\frac{s_{j}}{2}$, and $t_{j}<2u_{j}-s_{j}-\frac{d}{2}$ for $j\in\left\{1,2\right\}$, then the operator,
		\begin{align}
			S:=\langle\widehat{A_{0}}\rangle^{s_{1}}\langle c\nabla\rangle^{t_{1}}\left(H_{0}+1\right)^{-u_{1}}B\left(X\right)\langle\widehat{A_{0}}\rangle^{s_{2}}\p^{\delta}\left(H_{0}+1\right)^{-u_{2}},\ \delta\in\IN^{d},\ \left|\delta\right|\leq t_{2},
		\end{align}
		is densely defined on $\rg\left(\left.\left(H_{0}+1\right)^{u_{2}}\right|_{C_{c}^{\infty}\left(\IR^{d}\right)\otimes\mathcal{D}}\right)$ and admits a trace-class extension\break in $L^2\left(\IR^{d},H^{r}\right)$.
	\end{Corollary}
	
	\begin{proof}
		The space $\rg\left(\left.\left(H_{0}+1\right)^{u_{2}}\right|_{C_{c}^{\infty}\left(\IR^{d}\right)\otimes\mathcal{D}}\right)$ is dense in $L^2\left(\IR^{d},H^{r}\right)$ by Lemma \ref{densesetlem}. Let $f\in\rg\left(\left.\left(H_{0}+1\right)^{u_{2}}\right|_{C_{c}^{\infty}\left(\IR^{d}\right)\otimes\mathcal{D}}\right)$. Then there is $g\in C_{c}^{\infty}\left(\IR^{d}\right)\otimes\mathcal{D}$, such that $\left(H_{0}+1\right)^{-u_{2}}f=g$. Since $\langle\widehat{A_{0}}\rangle^{s_{2}}$ and $\p^{\delta}$ map $C_{c}^{\infty}\left(\IR^{d}\right)\otimes\mathcal{D}$ to itself, we obtain that $\langle\widehat{A_{0}}\rangle^{s_{2}}\p^{\delta}\left(H_{0}+1\right)^{-u_{2}}f\in C_{c}^{\infty}\left(\IR^{d}\right)\otimes\mathcal{D}$. Since $B\in L^{2,End}_{loc}\left(\IR^{d},\left(\mathcal{D},H^{r}\right)\right)$, Lemma \ref{densemultiplierdomlem} implies that
		\begin{align}\langle\widehat{A_{0}}\rangle^{s_{2}}\p^{\delta}\left(H_{0}+1\right)^{-u_{2}}f\in\dom\left(B\left(X\right)\right).
		\end{align}
	Therefore $\rg\left(\left.\left(H_{0}+1\right)^{u_{2}}\right|_{C_{c}^{\infty}\left(\IR^{d}\right)\otimes\mathcal{D}}\right)\subseteq\dom\left(S\right)$, which shows that $S$ is densely defined.
		
		By polar decomposition there exists a measurable family $\left(U\left(x\right)\right)_{x\in\IR^{d}}$ of unitary operators in $H^{r}$, such that for a.e. $x\in\IR^{d}$,
		\begin{align}
			B\left(x\right)=\left|B\left(x\right)^{\ast}\right|^{\frac{1}{2}}U\left(x\right)\left|B\left(x\right)\right|^{\frac{1}{2}}.
		\end{align}
		Accordingly we may decompose
		\begin{align}
			S&=Q_{1}Q_{2},\nonumber\\
			Q_{1}&=\langle\widehat{A_{0}}\rangle^{s_{1}}\langle c\nabla\rangle^{t_{2}}\left(H_{0}+1\right)^{-u_{1}}\left|B\left(X\right)^{\ast}\right|^{\frac{1}{2}},\nonumber\\
			Q_{2}&=U\left(X\right)\left|B\left(X\right)\right|^{\frac{1}{2}}\langle\widehat{A_{0}}\rangle^{s_{2}}\p^{\delta}\left(H_{0}+1\right)^{-u_{2}}.
		\end{align}
		For $Q_{2}$ we obtain, using the commutativity of $\widehat{A_{0}}$, $\p$ and $H_{0}$ on adequate domains and the $\ast$-invariance of Hilbert-Schmidt operators,
		\begin{align}
			&\left\|Q_{2}\right\|_{S^2\left(L^2\left(\IR^{d},H^{r}\right)\right)}=\left\|\left|B\left(X\right)\right|^{\frac{1}{2}}\langle\widehat{A_{0}}\rangle^{s_{2}}\p^{\delta}\left(H_{0}+1\right)^{-u_{2}}\right\|_{S^2\left(L^2\left(\IR^{d},H^{r}\right)\right)}\nonumber\\
			=&\left\|\langle\widehat{A_{0}}\rangle^{s_{2}}\p^{\delta}\left(H_{0}+1\right)^{-u_{2}}\left|B\left(X\right)\right|^{\frac{1}{2}}\right\|_{S^2\left(L^2\left(\IR^{d},H^{r}\right)\right)},
		\end{align}
		which implies that $Q_{2}$ posesses a Hilbert-Schmidt extension by Lemma \ref{traceliftlem}. Moreover Lemma \ref{traceliftlem} also implies that $Q_{1}$ admits a Hilbert-Schmidt extension, because $\ast$-invariance of trace-class operators yields,
		\begin{align}
			\int_{\IR^{d}}\left\|B\left(x\right)^{\ast}\right\|_{S^1\left(H^{r}\right)}\Id x=\int_{\IR^{d}}\left\|B\left(x\right)\right\|_{S^1\left(H^{r}\right)}\Id x<\infty.
		\end{align}
		Consequently $S$ admits a trace-class extension.
	\end{proof}

	The next result is concerned with giving conditions on a family $B$, such that the operator $\left(H_{0}+1\right)^{-1}B\left(X\right)\left(H_{0}+1\right)^{-N}$ admits a trace-class extension in $L^2\left(\IR^{d},H^{r}\right)$. To that end, we will use Lemma \ref{leibnizsplitlem} to rewrite $B\left(X\right)$, and obtain a decomposition of
	\begin{align}
		\left(H_{0}+1\right)^{-1}B\left(X\right)\left(H_{0}+1\right)^{-N}
\end{align}
into operators of the type discussed in Corollary \ref{traceliftcor}.

	\begin{Proposition}\label{traceprop}
		Denote $n:=\max\left(\lfloor\frac{d}{2}-2\rfloor+1,0\right)$. Let $\alpha\in\IR^{\geq 0}$ be fixed. Let $B\in W^{n,2,End}_{loc}\left(\IR^{d},\left(\mathcal{D},H^{r}\right)\right)$, and
		\begin{align}
			\int_{\IR^{d}}\left\|\p^{\gamma,End}\left(B\langle A_{0}\rangle^{-\alpha}\right)\left(x\right)\right\|_{S^1\left(H^{r}\right)}\Id x<\infty,
		\end{align}
		for all $\gamma\in\IN^{d}$ with $\left|\gamma\right|\leq n$.\\
		Then for $N>\frac{\alpha}{2}+\frac{n}{2}+\frac{d}{4}$, the operator
		\begin{align}
			\left(H_{0}+1\right)^{-1}B\left(X\right)\left(H_{0}+1\right)^{-N}
		\end{align}
		is densely defined on $\rg\left(\left.\left(H_{0}+1\right)^{N}\right|_{C_{c}^{\infty}\left(\IR^{d}\right)\otimes\mathcal{D}}\right)$, and admits a trace-class extension in $L^2\left(\IR^{d},H^{r}\right)$.
	\end{Proposition}
	
	\begin{proof}
		Let $f\in\rg\left(\left.\left(H_{0}+1\right)^{N}\right|_{C_{c}^{\infty}\left(\IR^{d}\right)\otimes\mathcal{D}}\right)$, which is dense in $L^2\left(\IR^{d},H^{r}\right)$ by Lemma \ref{densesetlem}. Then $g=\left(H_{0}+1\right)^{-N}f\in C_{c}^{\infty}\left(\IR^{d}\right)\otimes\mathcal{D}$. Lemma \ref{densemultiplierdomlem} then implies that
		\begin{align}
			\left(H_{0}+1\right)^{-1}B\left(X\right)\left(H_{0}+1\right)^{-N}
		\end{align}
		is densely defined. Let $2l+k=n$ with $k\in\left\{0,1\right\}$ and $l\in\IN$, and denote\break$B_{\gamma,\delta}^{k,l}\left(x\right):=C_{\gamma,\delta}^{k,l}\left(\p^{\gamma,End}C\right)\left(x\right)\langle A_{0}\rangle^{-\alpha}$ for $\gamma,\delta\in\IN^{d}$. Then Lemma \ref{leibnizsplitlem} implies 
		\begin{align}\label{tracepropeq1}
			&\left(H_{0}+1\right)^{-1}B\left(X\right)\left(H_{0}+1\right)^{-N}f=\left(H_{0}+1\right)^{-1}B\left(X\right)g\nonumber\\
			=&\left(H_{0}+1\right)^{-1}\left(ic\nabla+1\right)^{-k}\left(\Delta+1\right)^{-l}\sum_{\stackrel{\gamma,\delta\in\IN^{d}}{\left|\gamma+\delta\right|\leq n}}C_{\gamma,\delta}^{k,l}\left(\p^{\gamma,End}B\right)\left(X\right)\p^{\delta}g\nonumber\\
			=&\sum_{\stackrel{\gamma,\delta\in\IN^{d}}{\left|\gamma+\delta\right|\leq n}}\left(H_{0}+1\right)^{-1}\left(ic\nabla+1\right)^{-k}\left(\Delta+1\right)^{-l}C_{\gamma,\delta}^{k,l}\left(\p^{\gamma,End}B\right)\left(X\right)\p^{\delta}\left(H_{0}+1\right)^{-N}f\nonumber\\
			=&\langle c\nabla\rangle^{k}\left(1+ic\nabla\right)^{-k}\sum_{\stackrel{\gamma,\delta\in\IN^{d}}{\left|\gamma+\delta\right|\leq n}}\langle c\nabla\rangle^{-n}\left(H_{0}+1\right)^{-1}B_{\gamma,\delta}^{k,l}\left(X\right)\langle\widehat{A_{0}}\rangle^{\alpha}\p^{\delta}\left(H_{0}+1\right)^{-N}f.
		\end{align}
		The operator $\langle c\nabla\rangle^{k}\left(1+ic\nabla\right)^{-k}$ is bounded. For all $\gamma,\delta\in\IN^{d}$ with $\left|\gamma+\delta\right|\leq n$, the operators
		\begin{align}
			\langle c\nabla\rangle^{-n}\left(H_{0}+1\right)^{-1}B_{\gamma,\delta}^{k,l}\left(X\right)\langle\widehat{A_{0}}\rangle^{\alpha}\p^{\delta}\left(H_{0}+1\right)^{-N}
		\end{align}
		admit trace-class extensions, which follows by Corollary \ref{traceliftcor}, because
		\begin{align}
			B^{k,l}_{\gamma,\delta}\in L^{2,End}_{loc}\left(\IR^{d},\left(\mathcal{D},H^{r}\right)\right),
			\end{align}
			 and
		\begin{align}
			\int_{\IR^{d}}\left\|B^{k,l}_{\gamma,\delta}\left(x\right)\right\|_{S^1\left(H^{r}\right)}<\infty,
		\end{align}
		by assumption on $B$, and choosing $u_{1}=1$, $s_{1}=0$, $t_{1}=-n$, $u_{2}=N$, $s_{2}=\alpha$, and $t_{2}=n$. Consequently $\left(H_{0}+1\right)^{-1}B\left(X\right)\left(H_{0}+1\right)^{-N}$ admits a trace-class extension.
	\end{proof}

	By a complex interpolation argument, we may extend the trace-class membership of $\left(H_{0}+1\right)^{-1}B\left(X\right)\left(H_{0}+1\right)^{-N}$ to the operator $\left(H_{0}+1\right)^{-m-1}B\left(X\right)\left(H_{0}+1\right)^{-N+m}$.
	
	\begin{Corollary}\label{traceinterpolcor}
		Let $B$, $\alpha$, $n$, and $N$ satisfy the same prerequisites as in Proposition \ref{traceprop}. Additionally assume that $B^{\ast}\in W^{n,2,End}_{loc}\left(\IR^{d},\left(\mathcal{D},H^{r}\right)\right)$, and
		\begin{align}
			\int_{\IR^{d}}\left\|\p^{\beta,End}\left(B^{\ast}\langle A_{0}\rangle^{-\alpha}\right)\left(x\right)\right\|_{S^1\left(H^{r}\right)}\Id x<\infty.
		\end{align}
		Then for all $m\in\IN$ with $0\leq m\leq N-1$ the operators
		\begin{align}
			\left(H_{0}+1\right)^{-m-1}B\left(X\right)\left(H_{0}+1\right)^{-N+m}
		\end{align}
		are densely defined on $\rg\left(\left.\left(H_{0}+1\right)^{N}\right|_{C_{c}^{\infty}\left(\IR^{d}\right)\otimes\mathcal{D}}\right)$, and admit trace-class extensions in $L^2\left(\IR^{d},H^{r}\right)$.
	\end{Corollary}
	
	\begin{proof}
		First we show that
		\begin{align}
			\left.\left(H_{0}+1\right)^{-m-1}B\left(X\right)\left(H_{0}+1\right)^{-N+m}\right|_{\rg\left(\left.\left(H_{0}+1\right)^{N}\right|_{C_{c}^{\infty}\left(\IR^{d}\right)\otimes\mathcal{D}}\right)}
		\end{align}
		is densely defined. This follows by Lemma \ref{densesetlem} and Lemma \ref{densemultiplierdomlem} and if we note that $\left(H_{0}+1\right)^{m}$ maps $C_{c}^{\infty}\left(\IR^{d}\right)\otimes\mathcal{D}$ to itself.\\
		Now consider the operator $S_{0}:=\left(H_{0}+1\right)^{-1}B\left(X\right)\left(H_{0}+1\right)^{-1}$. $S_{0}$ is densely defined on
		\begin{align}
			\rg\left(\left.\left(H_{0}+1\right)\right|_{C_{c}^{\infty}\left(\IR^{d}\right)\otimes\mathcal{D}}\right)\subseteq\dom\left(S_{0}\right),
		\end{align}
		by Lemma \ref{densesetlem} and Lemma \ref{densemultiplierdomlem}. We also find
		\begin{align}\label{traceinterpolcoreq1}
			S_{0}^{\ast}&\supseteq\left(H_{0}+1\right)^{-1}\left(B\left(X\right)\right)^{\ast}\left(H_{0}+1\right)^{-1}\nonumber\\
			&\supseteq\left.\left(H_{0}+1\right)^{-1}B^{\ast}\left(X\right)\left(H_{0}+1\right)^{-1}\right|_{\rg\left(\left.\left(H_{0}+1\right)\right|_{C_{c}^{\infty}\left(\IR^{d}\right)\otimes\mathcal{D}}\right)},
		\end{align}
		which shows that $S_{0}^{\ast}$ is densely defined. Since $S_{0}$, $S_{0}^{\ast}$ are both densely defined, $S_{0}$ is closable and admits therefore a closure $S$.
		Denote $T=\left(H_{0}+1\right)^{N-1}$. Then by Corollary \ref{traceprop} the operator $S_{0}T^{-1}$ admits a trace-class extension $R$, which is thus bounded, therefore $R=\overline{S_{0}T^{-1}}$. Since $ST^{-1}$ is also a closed extension of $S_{0}T^{-1}$, we conclude $R=ST^{-1}$.
		Because $S_{0}^{\ast}=S^{\ast}$, and if we consider (\ref{traceinterpolcoreq1}), the operator $S^{\ast}T^{-1}$ is a closed extension of
		\begin{align}
			\left.\left(H_{0}+1\right)^{-1}B^{\ast}\left(X\right)\left(H_{0}+1\right)^{-N}\right|_{\rg\left(\left.\left(H_{0}+1\right)^{N}\right|_{C_{c}^{\infty}\left(\IR^{d}\right)\otimes\mathcal{D}}\right)},
		\end{align}
		which is densely defined and admits a trace-class extension $R'$ by Corollary \ref{traceliftcor}, which must be its closure. Therefore $S^{\ast}T^{-1}=R'$. We thus meet the prerequisites of an interpolation theorem (Theorem 3.2 \cite{GeLaSuTo}), ensuring that for $x\in\left[0,1\right]$ the operators $T^{-x}ST^{-1+x}$ defined on $\dom\left(T\right)$ are closable and $R_{x}=\overline{T^{-x}ST^{-1+x}}$ are trace-class operators. Now choose $x=\frac{m}{N-1}$, then
		\begin{align}
			\left.R_{\frac{m}{N-1}}\right|_{\rg\left(\left.\left(H_{0}+1\right)^{N}\right|_{C_{c}^{\infty}\left(\IR^{d}\right)\otimes\mathcal{D}}\right)}=\left.\left(H_{0}+1\right)^{-m-1}B\left(X\right)\left(H_{0}+1\right)^{-N+m}\right|_{\rg\left(\left.\left(H_{0}+1\right)^{N}\right|_{C_{c}^{\infty}\left(\IR^{d}\right)\otimes\mathcal{D}}\right)},
		\end{align}
		which finishes the proof.
	\end{proof}

	We have now all ingredients prepared to show the principal result of this work. Before stating the main theorem, let us summarize all accumulated conditions on the operator family $A$ in the following Hypothesis for convenience. Let us introduce the notation
	\begin{align}
		c\nabla^{End}A&:=\sum_{j=1}^{d}c\left(\Id x^{j}\right)\p_{x^{j}}^{End}A,\nonumber\\
		\nabla^{End}Ac&:=\sum_{j=1}^{d}\p_{x^{j}}^{End}Ac\left(\Id x^{j}\right).
	\end{align}

	\begin{Hypothesis}\label{mainassumpt}
		
		Denote $n:=\max\left(\lfloor\frac{d}{2}-2\rfloor+1,0\right)$. Let $\alpha,\beta\in\IR^{\geq 0}$ be fixed.
		
		Let $A_{0}$ be a self-adjoint operator in $H^{r}$, and let $\mathcal{D}=\bigcup_{n\in\IN}\rg\left(\one_{\left[-n,n\right]}\left(A_{0}\right)\right)$ and assume that $c\left(\Id x^{j}\right)\left(\mathcal{D}\right)\subseteq\dom\left(A_{0}\right)$, $j\in\left\{1,\ldots,d\right\}$.
		
		Let $A=\left(A\left(x\right)\right)_{x\in\IR^{d}}$ be a family of symmetric operators in $H^{r}$ with $\mathcal{D}\subseteq\dom\left(A\left(x\right)\right)$. If $A\left(x\right)$ commutes with $c\left(\Id x^{j}\right)$ for all $x\in\IR^{d}$ and $j\in\left\{1,\ldots,d\right\}$, then let $N\in\IN$ with $N>\frac{\alpha}{2}+\frac{n}{2}+\frac{d}{4}$, otherwise let $N>\max\left(\frac{\alpha}{2}+\frac{n}{2}+\frac{d}{4},\frac{\beta}{2}+\frac{n}{2}+\frac{d}{4}+\frac{1}{2}\right)$.
		
		Assume that $A\in W^{2N-1,2,End}_{loc}\left(\IR^{d},\left(\mathcal{D},H^{r}\right)\right)$, and
		\begin{align}\label{mainassumpteq1}
			\int_{\IR^{d}}\left\|\p^{\gamma,End}\left(c\nabla^{End}A\langle A_{0}\rangle^{-\alpha}\right)\left(x\right)\right\|_{S^1\left(H^{r}\right)}\Id x&<\infty,\nonumber\\
			\int_{\IR^{d}}\left\|\p^{\gamma,End}\left(\nabla^{End}Ac\langle A_{0}\rangle^{-\alpha}\right)\left(x\right)\right\|_{S^1\left(H^{r}\right)}\Id x&<\infty,\nonumber\\
			\int_{\IR^{d}}\left\|\p^{\gamma,End}\left(\left[c\left(\Id x^{j}\right),A\right]\langle A_{0}\rangle^{-\beta}\right)\left(x\right)\right\|_{S^1\left(H^{r}\right)}\Id x&<\infty,\ j\in\left\{1,\ldots,d\right\},
		\end{align}
		for all $\gamma\in\IN^{d}$ with $\left|\gamma\right|\leq n$. Assume that for $\gamma\in\IN^{d}$, $1\leq\left|\gamma\right|\leq 2N-1$, there exist $t\in\left[0,\frac{\left|\gamma\right|}{2}\right]$ and $p\in\left[2,+\infty\right]$ such that
		\begin{align}\label{mainassumpteq2}
			\left\|x\mapsto\left\|\left(\p^{\gamma,End}A\right)\left(x\right)\langle A_{0}\rangle^{-2t}\right\|_{B\left(H^{r}\right)}\right\|_{L^{p}\left(\IR^{d}\right)}<\infty,\ \text{for }\frac{\left|\gamma\right|}{2}-t>\frac{d}{2p},
		\end{align}
		or
		\begin{align}\label{mainassumpteq3}
			&\left\|x\mapsto\left\|\left(\p^{\gamma,End}A\right)\left(x\right)\left(A_{0}^{2}+z\right)^{-t}\right\|_{B\left(H^{r}\right)}\right\|_{L^{p}\left(\IR^{d}\right)}=o\left(1\right),\ z\to+\infty,\nonumber\\
			&\text{for }\left(\frac{\left|\gamma\right|}{2}=t, p=2\right)\vee\left(d\geq 3,\frac{d}{p}\in\IN,\frac{\left|\gamma\right|}{2}-t=\frac{d}{2p}\right).
		\end{align}
		Finally also assume
		\begin{align}\label{mainassumpteq4}
			\esssup_{x\in\IR^{d}}\left\|A_{0}^{k}\left(A\left(x\right)-A_{0}\right)\left(A_{0}^{2}+z\right)^{-\frac{k+1}{2}}\right\|_{B\left(H^{r}\right)}=o\left(1\right),\ z\to+\infty,\ 0\leq k\leq 2N-1.
		\end{align}
	\end{Hypothesis}

Since Hypothesis \ref{mainassumpt} comprises of several technical conditions on the family $A$, we shall discuss their necessity and differentiate their importance. Let us first consider the set of conditions (\ref{mainassumpteq2}), (\ref{mainassumpteq3}), and (\ref{mainassumpteq4}), which are chosen such that $\dom\left(\left(D^{\ast}D\right)^{N}\right)=\dom\left(\left(DD^{\ast}\right)^{N}\right)=W^{2N,2}\left(\IR^{d},H^{r}\right)\cap L^2\left(\IR^{d},\dom\left(A_{0}^{2N}\right)\right)$, see Proposition \ref{hdomainprop}. We may categorize these conditions as "minor" since they do not pertain to trace-class properties of $D$ but only of its domain. The essential conditions on $A$, central to trace-class properties of $D$, are given by (\ref{mainassumpteq1}), and they are in analogy to the conditions (\ref{pushcond}) in \cite{Push} and (\ref{gescond}) in \cite{GLMST}. We should note that the additional derivatives $\p^{\gamma,End}$ for $\gamma\leq n$ appear only in dimension $d\geq 4$, while the entire second line of (\ref{mainassumpteq1}) is not present for $d=1$, since w.l.o.g. Clifford multiplication is scalar and therefore commutes with any operator family. We also note that the requirement on the Clifford multiplication, $c\left(\Id x^{j}\right)\left(\mathcal{D}\right)\subseteq\dom\left(A_{0}\right)$, $j\in\left\{1,\ldots,d\right\}$, is always satisfied if $A_{0}$ and $c\left(\Id x^{j}\right)$ commute.

Let us illustrate the essential conditions by giving a simple example concerned with operator families generated by a perturbed Dirac operator on $\IR^{m}$, which is similar to the discussed case in \cite{Car} for a diagonal matrix potential.

	\begin{Example}\label{example}
		Denote $n:=\max\left(\lfloor\frac{d}{2}-2\rfloor+1,0\right)$.
		
		Let $m,s\in\IN$ and let $\widehat{c}$ be a Clifford multiplication over $\IR^{m}$ in $\IC^{s}$, and consider the self-adjoint Dirac operator $A_{0}:=\widehat{c}\ \widehat{\nabla}$ in $H^{r}$ for $H=L^2\left(\IR^{m},\IC^{s}\right)$.
		
		Let $\delta>\frac{m}{2}$, $\alpha>\frac{m}{2}+\delta$, and let $N\in\IN$ with $N>\frac{\alpha}{2}+\frac{n}{2}+\frac{d}{4}$.
		
		Let $v:\IR^{d}\times\IR^{m}\rightarrow\IR$ be a $C^{2N-1}$-function such that its derivatives are bounded, and that 
		\begin{align}\label{exampleeq1}
			\int_{\IR^{d}}\int_{\IR^{m}}\left|\p^{\gamma'}_{x}v\left(x,y\right)\right|\langle y\rangle^{\delta}\Id y\ \Id x<\infty,\ 1\leq\left|\gamma'\right|\leq n+1.
		\end{align}
		Define
		\begin{align}
			\left(V\left(x\right)f\right)\left(y\right)&:=v\left(x,y\right)f\left(y\right),\ x\in\IR^{d},\ y\in\IR^{m},\nonumber\\
			\dom\left(V\left(x\right)\right)&=H^{r}=L^2\left(\IR^{m},\IC^{s}\right)\otimes\IC^{r},\ x\in\IR^{d},\nonumber\\
			A\left(x\right)&:=A_{0}+V\left(x\right),\ x\in\IR^{d},\nonumber\\
			\dom\left(A\left(x\right)\right)&=\dom\left(A_{0}\right)=W^{1,2}\left(\IR^{m},\IC^{s}\right)\otimes\IC^{r},\ x\in\IR^{d}.
		\end{align}
		Because $v$ is smooth, we have $A\in W^{2N-1,2,End}_{loc}\left(\IR^{d},\left(\mathcal{D},H^{r}\right)\right)$. Furthermore we note that for all $\gamma\in\IN^{d}$ with $1\leq\left|\gamma\right|\leq 2N-1$ one obtains
	\begin{align}
		\sup_{x\in\IR^{d},y\in\IR^{m}}\left|\p^{\gamma}_{x}v\left(x,y\right)\right|=:C_{\gamma}<\infty,
	\end{align}
	which yields condition (\ref{mainassumpteq2}). Because $v$ is smooth with bounded derivatives, we also find
	\begin{align}
		&\sup_{x\in\IR^{d},y\in\IR^{m}}\left|\p^{\gamma}_{x}v\left(x,y\right)\right|\left\|\left(A_{0}^{2}+z\right)^{-\frac{k+1}{2}}\right\|_{B\left(H^{r}\right)}=C_{\gamma}z^{-\frac{k+1}{2}}\nonumber\\
		=&o\left(1\right),\ z\to+\infty, 0\leq\left|\gamma\right|=k\leq 2N-1,
	\end{align}
	which yields condition (\ref{mainassumpteq4}). The essential condition (\ref{mainassumpteq1}) reads
	\begin{align}
		\int_{\IR^{d}}\left\|\left(\p^{\gamma'}V\right)\left(x\right)\langle\widehat{c}\ \widehat{\nabla}\rangle^{-\alpha}\right\|_{S^1\left(H^{r}\right)}\Id x<\infty,\ 1\leq\left|\gamma'\right|\leq n+1,
	\end{align}
	which is satisfied according to (Corollary 4.8, \cite{Simon}) and (\ref{exampleeq1}). The second and third line of condition (\ref{mainassumpteq1}) is voided since $V$ commutes with $\widehat{c}$.
	
	The operator family $A$ therefore satisfies Hypothesis \ref{mainassumpt} and is admissible for the following Theorem \ref{mainthm}. It states that the associated Callias operator $D=ic\nabla+A\left(X\right)$ in this example satisfies
	\begin{align}
		\left(D^{\ast}D+1\right)^{-N}-\left(DD^{\ast}+1\right)^{-N}\in S^1\left(L^2\left(\IR^{d},H^{r}\right)\right)=S^1\left(L^2\left(\IR^{d+m},\IC^{rs}\right)\right).
	\end{align}
	\end{Example}

	We state the principal result of this work.
	
	\begin{Theorem}\label{mainthm}
		Assume Hypothesis \ref{mainassumpt}. Then
		\begin{align}
			\left(D^{\ast}D+1\right)^{-N}-\left(DD^{\ast}+1\right)^{-N}\in S^1\left(L^2\left(\IR^{d},H^{r}\right)\right).
		\end{align}
	\end{Theorem}
	
	\begin{proof}
		A direct consequence of Proposition \ref{hdomainprop} is that $\dom\left(\left(D^{\ast}D\right)^{k}\right)=\dom\left(\left(DD^{\ast}\right)^{k}\right)=\dom\left(H_{0}^{k}\right)$, for $0\leq k\leq N$. Thus by the resolvent identity, we find 
		\begin{align}
			&\left(D^{\ast}D+1\right)^{-N}-\left(DD^{\ast}+1\right)^{-N}=\sum_{k=0}^{N-1}\left(DD^{\ast}+1\right)^{-k-1}\left[D,D^{\ast}\right]\left(D^{\ast}D+1\right)^{-N+k}\nonumber\\
			=&\sum_{k=0}^{N-1}B_{k}\left(H_{0}+1\right)^{-k-1}\left[D,D^{\ast}\right]\left(H_{0}+1\right)^{-N+k}\widetilde{B}_{N-k},
		\end{align}
		where the operators
		\begin{align}
			B_{k}&:=\overline{\left(DD^{\ast}+1\right)^{-k-1}\left(H_{0}+1\right)^{k+1}},\nonumber\\
			\widetilde{B}_{k}&:=\overline{\left(H_{0}+1\right)^{k}\left(D^{\ast}D+1\right)^{-k}},
		\end{align}
		are bounded. Thus it suffices to show that
		\begin{align}
			\left(H_{0}+1\right)^{-k-1}\left[D,D^{\ast}\right]\left(H_{0}+1\right)^{-N+k}\in S^{1}\left(L^2\left(\IR^{d},H^{r}\right)\right), 0\leq k\leq N-1.
		\end{align}
		For $0\leq k\leq N-1$ and $\phi\in\rg\left(\left.\left(H_{0}+1\right)^{N}\right|_{C_{c}^{\infty}\left(\IR^{d}\right)\otimes\mathcal{D}}\right)$ one obtains
		\begin{align}
			&\left(H_{0}+1\right)^{-k-1}\left[D,D^{\ast}\right]\left(H_{0}+1\right)^{-N+k}\phi=2i\left(H_{0}+1\right)^{-k-1}\left[c\nabla,A\left(X\right)\right]\left(H_{0}+1\right)^{-N+k}\phi\nonumber\\
			=&2i\sum_{j=1}^{d}\left(H_{0}+1\right)^{-k-1}\left[c\left(\Id x^{j}\right),A\right]\left(X\right)\p_{x^{j}}\left(H_{0}+1\right)^{-N+k}\phi\nonumber\\
			&+2i\sum_{j=1}^{d}\left(H_{0}+1\right)^{-k-1}c\left(\Id x^{j}\right)\left(\p^{End}_{x^{j}}A\right)\left(X\right)\left(H_{0}+1\right)^{-N+k}\phi\nonumber\\
			=&2i\left(H_{0}+1\right)^{-k-1}\sum_{j=1}^{d}\left[c\left(\Id x^{j}\right),A\right]\left(X\right)\p_{x^{j}}\left(H_{0}+1\right)^{-N+k}\phi\nonumber\\
			&+2i\left(H_{0}+1\right)^{-k-1}\left(c\nabla^{End}A\right)\left(X\right)\left(H_{0}+1\right)^{-N+k}\phi,
		\end{align}
		where we use that the conditions posed on $A$ imply that $\dom\left(A\left(x\right)\right)=\dom\left(A_{0}\right)$, for a.e. $x\in\IR^{d}$ (see Remark \ref{adomsamerem}), and that $c\left(\Id x^{j}\right)\left(\mathcal{D}\right)\subseteq\dom\left(A_{0}\right)$, $j\in\left\{1,\ldots,d\right\}$. Since $\p_{x^{j}}\left(H_{0}+1\right)^{-\frac{1}{2}}$ is bounded for $j\in\left\{1,\ldots,d\right\}$, it suffices to show that
		\begin{align}
			&\left(H_{0}+1\right)^{-k-1}\left[c\left(\Id x^{j}\right),A\right]\left(X\right)\left(H_{0}+1\right)^{-N+k+\frac{1}{2}},\ j\in\left\{1,\ldots,d\right\},\nonumber\\
			&\left(H_{0}+1\right)^{-k-1}\left(c\nabla^{End}A\right)\left(X\right)\left(H_{0}+1\right)^{-N+k},
		\end{align}
		admit trace-class extensions for $0\leq k\leq N-1$, which is a direct consequence of Corollary \ref{traceinterpolcor}.
	\end{proof}

	The above presented Theorem \ref{mainthm} should provide a large enough store of operator families $A$ such that one may investigate trace formulas and spectral shift functions of the pair $\left(D^{\ast}D, DD^{\ast}\right)$, which might lead to new results for the index of the Callias operator $D$ in this abstract setup.

	\textsc{O. F\"urst, Institut f\"ur Analysis, Leibniz Universit\"at Hannover, Welfengarten 1, 30167 Hannover, Germany}\\
	E-Mail: \url{fuerst@math.uni-hannover.de}


\begin{thebibliography}{99}
		
		
	\bibitem{Adams} R. Adams, J. Fournier, \textit{Sobolev Spaces}, Amsterdam: Elsevier (2003).
	
	\bibitem{APS1} M. Atiyah, V. Patodi, I. Singer, \textit{Spectral asymmetry and Riemannian geometry.
	I.}, Math. Proc. Cambridge Philos. Soc. \textbf{77}, 43-69, (1975).

	\bibitem{APS2} M. Atiyah, V. Patodi, I. Singer, \textit{Spectral asymmetry and Riemannian geometry.
	II.}, Math. Proc. Cambridge Philos. Soc. \textbf{78}, 405-432, (1975).

	\bibitem{APS3} M. Atiyah, V. Patodi, I. Singer, \textit{Spectral asymmetry and Riemannian geometry.
	III.}, Math. Proc. Cambridge Philos. Soc. \textbf{79}, 71-99, (1976).
	
	\bibitem{Cal} C. Callias, \textit{Axial Anomalies and Index Theorems on Open Spaces}, Commun. Math.
	Phys. \textbf{62}, 213-234, (1978).
		
	\bibitem{Car} A. Carey, F. Gesztesy, G. Levitina, R. Nichols, F. Sukochev, D. Zanin, \textit{The Limiting Absorption Principle for Massless Dirac Operators, Properties of Spectral Shift Functions, and an Application to the Witten Index of Non-Fredholm Operators}, arXiv:2105.03024, (2021).
	
%
	
	\bibitem{GLMST} F. Gesztesy, Y. Latushkin, K. Makarov, F. Sukochev, Y. Tomilov, \textit{The Index Formula and the Spectral Shift Function for Relatively Trace Class Perturbations}, Adv.Math. \textbf{227}, 319-420, (2011).
	
	\bibitem{GeLaSuTo} F. Gesztesy, Y. Latushkin, F. Sukochev, Y. Tomilov, \textit{Some Operator Bounds Employing Complex Interpolation Revisited}, Operator Semigroups Meet Complex Analysis, Harmonic Analysis and Mathematical Physics. Operator Theory: Advances and Applications, Vol. \textbf{250}. Birkhäuser, Cham, 213-239, (2015).
	
	\bibitem{Kato} T. Kato, \textit{Perturbation Theory for Linear Operators}, Springer, Berlin (1980).
	
	\bibitem{Push} A. Pushnitski, \textit{The Spectral Flow, the Fredholm Index, and the Spectral Shift Function}, Spectral Theory of Differential Operators, 141-155,  (2008).
%
%

	\bibitem{Simon} B. Simon, \textit{Trace Ideals and Their Applications}, Mathematical Surveys and Monographs, Vol. \textbf{120}, Heidelberg: Am. Math. Soc. (2005).
	
	
	
	\end{thebibliography}
\end{document}